\documentclass[12pt,a4paper]{article}
\usepackage{amsfonts}
\usepackage{amsmath}
\usepackage{amssymb}
\usepackage{fullpage}
\usepackage{amsthm}
\usepackage{graphicx}

\newtheorem{theorem}{Theorem}[section]
\newtheorem{proposition}[theorem]{Proposition}
\newtheorem{lemma}[theorem]{Lemma} 
\newtheorem{corollary}[theorem]{Corollary}
\newtheorem{remark}[theorem]{Remark}

\newtheorem{definition}[theorem]{Definition}

\newcommand{\upw}{{\uparrow}}
\newcommand{\downw}{{\downarrow}}

\newcommand{\Lk}{{\mathrm{Lk}}}
\newcommand{\Star}{{\mathrm{St}}}
\newcommand{\link}{\Lk}
\newcommand{\ttt}{\mathbb{T}}
\newcommand{\rr}{{\mathbb R}}
\newcommand{\nn}{{\mathbb N}}
\newcommand{\zz}{{\mathbb Z}}
\newcommand{\cc}{{\mathbb C}}
\newcommand{\qq}{{\mathbb Q}}
\newcommand{\sss}{{\mathbb S}}
\newcommand{\semi}{{\rtimes}}
\newcommand{\id}{{\mathrm {Id}}}
\newcommand{\cR}{{\mathcal R}}
\newcommand{\cL}{{\mathcal L}}
\newcommand{\cN}{{\mathcal N}}
\newcommand{\cC}{{\mathcal C}}
\newcommand{\ug}{{\mathbf{g}}}
\newcommand{\ua}{{\mathbf{a}}}
\newcommand{\dirlim}{\lim}
\newcommand{\scw}{{\scriptstyle{W}}}
\newcommand{\scn}{{\scriptstyle{N}}}
\newcommand{\sce}{{\scriptstyle{E}}}
\newcommand{\scs}{{\scriptstyle{S}}}
\newcommand{\Hom}{\mathrm {Hom}}
\newcommand{\bigast}{\mathop{\scalebox{1.7}{\raisebox{-0.2ex}{$\ast$}}}}

\newcommand\mapright[1]{\smash{\mathop{\longrightarrow}\limits^{#1}}}
\newcommand\mapdown[1]{\Big\downarrow\rlap{$\vcenter{\hbox{$\scriptstyle#1$}}$}}

\title{Uncountably many groups of type $FP$} 
\author{Ian J. Leary\thanks{Partially supported by a Research Fellowship 
from the Leverhulme Trust.  Some of this work was done at MSRI, Berkeley, 
where research is supported by 
the National Science Foundation under Grant No.~DMS-1440140.}}

\date{\today}

\begin{document} 

\maketitle

\begin{abstract} 
We construct uncountably many discrete groups of type $FP$; 
in particular we construct groups of type $FP$ that do not 
embed in any finitely presented group.  We compute the 
ordinary, $\ell^2$, and compactly-supported cohomology of 
these groups.  For each $n\geq 4$ we construct a closed 
aspherical $n$-manifold that admits an uncountable family of 
acyclic regular coverings with non-isomorphic covering groups.  
\end{abstract}

\section{Introduction} 

An Eilenberg-Mac~Lane space or $K(G,1)$ for a group $G$ is a connected
CW-complex whose fundamental group is $G$ and whose universal cover is
contractible.   Equivalently a $K(G,1)$ is an aspherical CW-complex whose
fundamental group is $G$.  The group $G$ is of type $F$ if there is a
finite $K(G,1)$, and is of type $F_n$ if there is a $K(G,1)$ with
finite $n$-skeleton.  These conditions are generalizations of two 
standard group-theoretic conditions: $G$ can be finitely generated 
if and only if $G$ is type $F_1$, and $G$ can be finitely presented 
if and only if $G$ is type~$F_2$.  

There are a number of other finiteness conditions on groups that are
implied by type~$F$ or type~$F_n$.  The universal covering space of 
a $K(G,1)$ is a contractible free $G$-CW-complex, with one $G$-orbit
of cells for each cell in the given $K(G,1)$.  For a non-trivial ring 
$R$, a group $G$ is type $FH(R)$ if there is an $R$-acyclic free 
$G$-CW-complex with finitely many orbits of cells.  Similarly, 
$G$ is type~$FL(R)$ (resp.~$FP(R)$) if there is a finite resolution 
of $R$ by finitely generated free (resp.~finitely generated projective) 
$RG$-modules.  There are corresponding analogues of the properties 
$F_n$.  If $G$ is type $F$ then $G$ satisfies each of these other 
finiteness conditions for any $R$.  It has long been known that 
$G$ is type~$F$ if and only if $G$ is both $F_2$ and $FL(\zz)$.  

The question of whether $FP(\zz)$ implies $F_2$ was 
solved by Bestvina and Brady~\cite{BB}.  
To each finite flag complex $L$, they associated a group $BB_L$ in 
such a way that the finiteness properties of $BB_L$ were 
controlled by properties of $L$.  This gave a systematic way 
to construct groups enjoying different finiteness properties. 
In particular, in the case when $L$ is acyclic but is not 
contractible, they showed that $BB_L$ is $FP(\zz)$, and even 
$FH(\zz)$,  but is not $F_2$.  

Since there are only countably many finite presentations of groups,
there are only countably many finitely presented groups up to
isomorphism.  On the other hand it is well-known that there are
uncountably many finitely generated groups.  Since a countable group
contains at most countably many finitely generated subgroups, it
follows that not every finitely generated group is a subgroup of a
finitely presented group~\cite{neumann}.  Higman showed that a group
$G$ embeds in a finitely presented group if and only if $G$ admits a
recursive presentation, i.e., a presentation in which the sets of
generators and relators are recursively enumerable~\cite{higman}.

Once one knows that there are $FP_2(\zz)$ groups that are not finitely
presented, it becomes natural to ask how many there are.  There are
only countably many finite flag complexes $L$, and hence only
countably many Bestvina-Brady groups $BB_L$.  Moreover each $BB_L$ is
defined as a normal subgroup of a right-angled Artin group $A_L$ and
this group is of type~$F$ for every finite $L$.  Thus the examples
constructed in~\cite{BB} leave open the question of whether every
group of type $FP_2(\zz)$ embeds in a finitely presented group, and
whether there are uncountably many groups of type $FP_2(\zz)$.

We answer these questions; we construct an uncountable family of 
groups of type $FP(\zz)$, and determine which of these groups
are isomorphic to subgroups of finitely presented groups.  Our 
construction uses the techniques introduced in~\cite{BB}.  For 
any connected finite flag complex $L$ and any subset $S$ of $\zz$
we construct a group $G_L(S)$.  Moreover, the construction 
is functorial in $S$, in the sense that for $S\subseteq T\subseteq \zz$
there is a natural surjective group homomorphism $G_L(S)\rightarrow 
G_L(T)$.  The finiteness properties of $G_L(S)$ are controlled by 
properties of $L$ and its universal cover $\widetilde L$.  
Our main example, an uncountable family 
$G_L(S)$ of groups of type $FP(\zz)$, arises from the case when 
$L$ and $\widetilde L$ are both acyclic but $L$ is not contractible. 
The groups $G_L(S)$ can be viewed as interpolating between the 
groups $G_L(\emptyset)$ and $G_L(\zz)$, which were known previously:
$G_L(\zz)$ is $BB_L$ and $G_L(\emptyset)$ is isomorphic to the 
natural semidirect product $BB_{\widetilde L}\semi \pi_1(L)$.  

Just as in~\cite{BB}, the proofs of our main results are geometric,
involving Morse functions on CAT(0) cubical complexes.  For each $S$
and $L$, the group $G_L(S)$ acts as a group of automorphisms of a CAT(0)
cubical complex $X_L^{(S)}$.  This cubical complex is a branched cover
of the complex $X_L$ that plays a crucial role in~\cite{BB}, and the 
Morse function on $X_L$ induces a Morse function on $X_L^{(S)}$.  
Just as in~\cite{BB}, the group $G_L(S)$ acts freely cocompactly 
on a level set for this Morse function, and Morse theory is used 
to establish the connectivity properties of this level set.  In 
contrast to the groups $BB_L$ in~\cite{BB}, the groups $G_L(S)$ 
for general $S$ 
cannot be described as subgroups of other previously known groups.  
We give a number of ways of describing $G_L(S)$, each of which will 
play a role in our proofs.  Some of these ways simplify if either 
$0\in S$ or if $L$ has nlcp, which stands for `no local cut points', 
a property defined below in Definition~\ref{defnnlcp}.  
In particular $G_L(S)$ is isomorphic to: 

\begin{itemize} 

\item{} the group of deck transformations of the regular branched covering  
$X_L^{(S)}\rightarrow X_L/BB_L$; 

\item{} the group of deck transformations of the regular covering 
$X_t^{(S)}\rightarrow X_t/BB_L$, where $X_t^{(S)}$ denotes a non-integer
level set in $X_L^{(S)}$ and similarly $X_t$ in $X_L$;  

\item{} the fundamental group of a space
  obtained from $X_L/BB_L$ by removing some vertices, in the case when
  $L$ has nlcp;

\item{} the group given by a presentation $P_L(\Gamma,S)$, in the case
  when $0\in S$.  

\end{itemize} 

In Theorem~\ref{thmexess} we will \emph{define} $G_L(S)$ as the group of 
deck transformations of the regular branched covering
$X_L^{(S)}\rightarrow X_L/BB_L$.  

However, defining the CAT(0) cubical complex $X_L^{(S)}$ will occupy a
significant proportion of our paper.  For this reason, we give the
presentations $P_L(\Gamma,S)$ before stating our main theorems.  These
presentations realize the groups $G_L(S)$ and the homomorphisms
$G_L(S)\twoheadrightarrow G_L(T)$ for $0\in S\subseteq T$, in the
sense that the generating set depends only on $L$ while the relators
in $P_L(\Gamma,S)$ are a subset of the relators in $P_L(\Gamma,T)$.
Since there are isomorphisms $G_L(S)\cong G_L(T)$ whenever $T=S+n$ is
a translate of $S$, the presentations $P_L(\Gamma,S)$ describe the
isomorphism type of each $G_L(S)$ except for $G_L(\emptyset)$, which 
can be described as a semidirect product $BB_{\widetilde L}\semi\pi_1(L)$.  

Recall that a directed edge in a simplicial complex
is an ordered pair $(v,v')$ of vertices such that the corresponding
unordered pair is an edge.  A directed loop $\gamma=(a_1,\ldots,a_l)$
of length $l$ is an ordered $l$-tuple of directed edges whose
endpoints match up, in the sense that there are vertices $v_0,\ldots,v_l$ 
with $v_0=v_l$ and $a_i=(v_{i-1},v_{i})$ for each $i$.  
In the following definition, we assume that $L$ is a 
finite connected flag complex, and that $\Gamma$ is a finite 
collection of directed edge loops in $L$ that normally generates 
$\pi_1(L)$; equivalently if one attaches discs to $L$ along the loops
in $\Gamma$ one obtains a simply-connected complex.  

\begin{definition}
For $L$ and $\Gamma$ as above and for any set $S$ with 
$0\in S\subseteq \zz$, the presentation $P_L(\Gamma,S)$ 
has as generators the directed edges of $L$, subject to the following 
relations.  

\begin{itemize} 

\item 
(Edge relations.) 
For each directed edge $a=(x,y)$ with opposite edge $\overline a=(y,x)$, 
the relation $a\overline a=1$; 

\item 
(Triangle relations.) 
For each directed triangle $(a,b,c)$ in $L$, the relations 
$abc=1$ and $a^{-1}b^{-1}c^{-1}=1$; 

\item 
(Long cycle relations.) 
For each $n\in S- \{0\}$, and each $(a_1,a_2,\ldots,a_l)\in \Gamma$, 
the relation $a_1^na_2^n\cdots a_l^n=1$. 
\end{itemize} 
\end{definition} 

In the case when $L$ is simply connected, we may take $\Gamma$ to be empty, 
in which case $P_L(\Gamma,S)$ is independent of $S$.  This reflects the 
fact that when $L$ is simply-connected, the natural map $G_L(S)\rightarrow 
G_L(\zz)\cong BB_L$ is an isomorphism for each $S$.  Our first main
result addresses all other cases.  

\begin{theorem}\label{thma}
Let $L$ be a fixed finite connected flag complex that is not 
simply-connected.  

\begin{enumerate}

\item{}
There are uncountably many (in fact $2^{\aleph_0}$) isomorphism 
types of group $G_L(S)$.   

\item{}
For $0\in S$, and any $\Gamma$ that normally generates $\pi_1(L)$, 
the presentation $P_L(\Gamma,S)$ is a presentation of the group
$G_L(S)$.  

\item{}
$G_L(S)$ is finitely presentable if and only if $S$ is finite.  

\item{}
$G_L(S)$ is isomorphic to a subgroup of a finitely presented group 
if and only if $S$ is recursively enumerable.  
\end{enumerate} 

\end{theorem}  

Our second main result summarizes the finiteness properties of the
groups $G_L(S)$.  Some of our results concerning the cases when 
either $S$ or $\zz-S$ is finite are omitted to simplify the
statement.  All finiteness properties mentioned in the 
statement will be defined in a subsequent section.  

\begin{theorem} \label{thmb}
Let $L$ be a connected finite flag complex, let $R$ be any ring 
in which $1\neq 0$, and let $S_0$ be any subset of 
$\zz$ such that both $S_0$ and $\zz- S_0$ are 
infinite.  The following are equivalent: 
\begin{itemize} 
\item 
$L$ and $\widetilde L$ are $R$-acyclic; 

\item
For each $S\subseteq \zz$, $G_L(S)$ is type $FH(R)$; 

\item 
For each $S\subseteq \zz$, $G_L(S)$ is type $FP(R)$; 

\item
$\pi_1(L)$ and $G_L(S_0)$ are type $FP(R)$.  

\end{itemize} 

For each $n\geq 2$, the following are equivalent: 
\begin{itemize} 
\item 
$L$ and $\widetilde{L}$ are $(n-1)$-$R$-acyclic; 

\item
For each $S\subseteq \zz$, $G_L(S)$ is type $FH_n(R)$;   

\item 
For each $S\subseteq \zz$, $G_L(S)$ is type $FP_n(R)$; 

\item
$\pi_1(L)$ and $G_L(S_0)$ are type $FP_n(R)$.  

\end{itemize} 
\end{theorem}

Bestvina-Brady groups have been studied extensively, and many of the
results that have been obtained for $BB_L$ have analogues for
$G_L(S)$.  In particular, our proof of Theorem~\ref{thmb} uses some techniques
from~\cite{BB,buxgon}.  For $G_L(S)$ we prove some of the results that
have been proved for $BB_L$ in~\cite{davispd,do,DL,LS}.  With some
hypotheses on $L$, we compute the ordinary cohomology of $G_L(S)$ in
terms of that of $BB_L$, which was calculated in~\cite{LS}, and we
compute the compactly supported and $\ell^2$-cohomology of $G_L(S)$
following~\cite{do}.  We use Davis's trick~\cite{davispd,davisbook} to
construct, for each $n\geq 4$, a closed aspherical $n$-manifold that
admits an uncountable family of acyclic regular covers with
non-isomorphic groups of deck transformations.  We also give a second
proof that $G_L(S)$ is $FP_2(\zz)$ whenever $\pi_1(L)$ is perfect that
does not use CAT(0) cubical complexes, along the lines
of~\cite{DL}.

The paper is structured as follows.  Sections 2--5 describe background
material concerning flag complexes, finiteness conditions, Artin
groups and Morse functions respectively.  Sections 6--8 concern links
in CAT(0) cubical complexes, especially the complex $X_L$, and
Sections 9--10 contain the construction of $X_L^{(S)}$, its Morse
function, and the connectivity of its level sets.  Sections 11--13
complete the proof of Theorem~\ref{thmb}; they discuss realizing $G_L(S)$ in
terms of fundamental groups, the definition of `sheets' in $X_L^{(S)}$
and the application of Brown's criterion respectively.  Section~14
gives presentations for $G_L(S)$, and Section~15 introduces a
set-valued invariant for a group and a finite sequence of elements,
which may be of independent interest.  These two sections complete the
proof of Theorem~\ref{thma}.  Sections 16--18 give analogues for $G_L(S)$ of
various results that have been obtained for $BB_L$.  In more detail,
Section~16 considers classifying spaces and ordinary cohomology;
Section 17 considers compactly supported and $\ell^2$-cohomology;
Section~18 establishes the existence of uncountably many (non-finitely
presented) Poincar\'e duality groups.  Sections 19~and~20 describe
alternative proofs of some of the results: Section~19 gives a proof
that $G_L(S)$ is $FP_2(\zz)$ whenever $\pi_1(L)$ is perfect that does not
use $X_L^{(S)}$, and Section~20 describes the homotopy type of level
sets in $X_L^{(S)}$ in the case when $\pi_1(L)$ is finite.  Finally,
Section~21 discusses a small number of open problems.

The author first heard the question of whether there can be uncountably 
many groups of type $FP$ from a conversation with Martin Bridson, Peter 
Kropholler, Robert Kropholler and Charles Miller~III.  Martin Bridson 
and the author subsequently discussed the problem, and this work benefitted 
from these discussions and a conversation with Ashot Minasyan.  The author 
thanks these five for their helpful comments on this work.  

\section{Examples of flag complexes} 

To apply Theorems A~and~B, we require certain finite flag complexes $L$. 
The purpose of this section is to establish that such complexes exist.  

A \emph{flag complex} or \emph{clique complex} is a simplicial complex
with the property that any finite set of mutually adjacent vertices
spans a simplex.  The realization of any poset is a flag complex and
so in particular the barycentric subdivision of any simplicial complex
is a flag complex.

If $L$ is connected, then $\widetilde{L}$ is acyclic if and only if
$L$ is aspherical, i.e., if and only if $L$ is an Eilenberg-Mac~Lane
space.  A theorem of Baumslag-Dyer-Heller tells us that for any finite
complex $K$, there is a finite Eilenberg-Mac~Lane space $L$ and a
homology isomorphism $L\rightarrow K$~\cite{BDH,mkt}.  Thus the
assumption that $\widetilde{L}$ is acyclic places no restrictions on
the homology of $L$.  This provides a source of examples for applying
our main theorems.  In particular if $L$ is the Eilenberg-Mac~Lane
space for a non-trivial acyclic group, then Theorems A~and~B imply
that the groups $G_L(S)$ comprise uncountably many isomorphism types
of groups of type $FP(\zz)$.

For a more explicit example, consider the presentation 2-complex 
for Higman's group
\[\langle a,b,c,d\,\, :\,\, a^b=a^2,\, b^c=b^2, \,
c^d=c^2,\, d^a=d^2\rangle.\]
This is a polygonal cell complex consisting of one 0-cell, 
four 1-cells and four pentagonal 2-cells.  
This 2-complex is both acyclic and aspherical, and has non-trivial 
fundamental group~\cite{BDH}.  Any flag triangulation of this
2-complex has nlcp (as in Definition~\ref{defnnlcp}) and can be used in
our main theorems.  For example, the second barycentric subdivision 
of the given polygonal structure is a suitable flag complex $L$ 
consisting of 240 triangles, 336 edges and 97 vertices.  
For other examples of 2-complexes that are acyclic and aspherical, 
see~\cite[section~3]{mkt} or~\cite[section~4]{bg}.

\section{Finiteness properties for groups} 

We briefly recall the definitions of some of the homological
finiteness properties for groups and give some results that we will
use to establish these properties.  The algebraic properties were
introduced by Bieri~\cite{bieri,brown}, $F$ and $F_n$ by
Wall~\cite{wall}, and the $FH$ and $FH_n$ properties
by Bestvina-Brady~\cite{BB}.  Recall that a CW-complex $X$ is said to
be $R$-acyclic if its reduced homology groups satisfy 
$\overline{H}_i(X;R)=0$ for all $i\geq -1$.  Similarly, a CW-complex 
$X$ is said to be $n$-$R$-acyclic if $\overline{H}_i(X;R)=0$ for $i\leq n$.

\begin{definition}\label{deffinone} 
Let $G$ be a discrete group, and let $R$ be a nontrivial ring.    
\begin{itemize} 
\item $G$ is type $F$ if there is a finite $K(G,1)$, or 
equivalently there is a contractible free $G$-CW-complex 
with finitely many orbits of cells.  

\item $G$ is type $FH(R)$ if there is an $R$-acyclic free
  $G$-CW-complex with finitely many orbits of cells.  

\item $G$ is type $FL(R)$ if there is some $n\geq 0$ and a 
long exact sequence of $RG$-modules 
\[0\rightarrow P_n\rightarrow P_{n-1}\rightarrow 
\cdots \rightarrow P_0\rightarrow R\rightarrow 0\]
in which each $P_i$ is a finitely generated free module.  

\item $G$ is type $FP(R)$ if there is a long exact sequence
as above in which each $P_i$ is a finitely generated projective 
module. 
\end{itemize}  
\end{definition} 

\begin{definition}\label{deffintwo} 
Let $G$ be a discrete group, $R$ a non-trivial ring, and $n\geq 0$ 
an integer.  
\begin{itemize} 
\item
$G$ is type $F_n$ if there is a $K(G,1)$ with finite $n$-skeleton.  

\item 
$G$ is type $FH_n(R)$ if there is an $(n-1)$-$R$-acyclic free 
$G$-CW-complex with finitely many orbits of cells.  

\item 
$G$ is type $FL_n(R)$ if there is a long exact sequence of
  $RG$-modules 
\[\rightarrow P_m\rightarrow P_{m-1}\rightarrow 
\cdots \rightarrow P_0\rightarrow R\rightarrow 0\]
in which $P_i$ is a finitely generated free module for $i\leq n$.  

\item $G$ is type $FP_n(R)$ if there is a long exact sequence
as above in which $P_i$ is a finitely generated projective 
module for $i\leq n$.  
\end{itemize}  
\end{definition} 

A contractible space is $R$-acyclic for any $R$, the cellular 
chain complex for a free $G$-CW-complex is a chain complex of 
free modules, and free modules are projective.  From these 
observations it follows that for any $R$, 
$F\implies FH(R)\implies FL(R)\implies FP(R)$, and 
similarly $F_n\implies FH_n(R)\implies FL_n(R)\implies FP_n(R)$.  
It can be shown that $FL_n(R)$ is equivalent to 
$FP_n(R)$, for each $n$ and $R$.  It follows from the universal 
coefficient theorem that whenever there is a ring homomorphism 
$R\rightarrow S$, $FX(R)\implies FX(S)$, where $FX(-)$ denotes 
any of the above finiteness conditions that involves a ring.  
In particular, $FX(\zz)\implies FX(R)$ for any $R$.  For any 
non-trivial ring $R$, $FP_1(R)$ is equivalent to $F_1$, and 
for $n\geq 2$, $FP_n(\zz)$ together with $F_2$ is equivalent 
to $F_n$~\cite[VIII.7.1]{brown}.  

$G$ is said to have cohomological dimension at most $n$ over 
$R$ if there is a projective resolution of $R$ as an $RG$-module 
of length $n$: 
\[0\rightarrow P_n\rightarrow P_{n-1}\rightarrow 
\cdots \rightarrow P_0\rightarrow R\rightarrow 0.\] If $G$ is type
$FP(R)$, then $G$ has finite cohomological dimension, and conversely
if $G$ has cohomological dimension $n$ over $R$ and $G$ is $FP_n(R)$
then $G$ is $FP(R)$.  If $H\leq G$, then the cohomological dimension
of $H$ is a lower bound for that of $G$.  Since non-trivial finite
groups have infinite cohomological dimension, it follows that every
group of type $FP(\zz)$ is torsion-free.

For a general ring $R$, the difference between $FP(R)$ and $FL(R)$ is
well understood in terms of $K$-theory; in particular if every
finitely generated projective $RG$-module is stably free then $G$ is
$FP(R)$ if and only if $G$ is $FL(R)$.  No examples are known that
distinguish between the properties $FP(\zz)$, $FL(\zz)$ and $FH(\zz)$.
For larger rings however, there is a difference.  If $G$ is any finite
group, then $G$ is $FP(\qq)$ by Maschke's theorem, but $G$ is not
$FL(\qq)$ unless $G$ is the trivial group.  No examples are known to
distinguish $FL(\qq)$ and $FH(\qq)$, but again there is a difference
for larger fields.  If $K=\qq[e^{2\pi i/3}]$ is the field obtained by
adjoining a cube root of 1 to~$\qq$, there is a (3-dimensional
crystallographic group) $G$ which is $FL(K)$, $FP(\qq)$ and not
$FL(\qq)$.  Since $FH(K)\Leftrightarrow FH(\qq)$ it follows that this
group cannot be $FH(K)$~\cite{epdn}.

A sequence $\rightarrow A_m\rightarrow A_{m+1}\rightarrow $ of abelian
groups and morphisms indexed by $\nn$ is \emph{essentially trivial} if 
for all $m$ there exists $m'>m$ so that the composite map $A_m\rightarrow
A_{m'}$ is trivial.  The following statement is a simplified version 
of K.~S. Brown's criterion~\cite{browncrit}, which will suffice for 
our purposes.  

\begin{theorem}\label{brownscrit}
Suppose that $X$ is a finite-dimensional $R$-acyclic $G$-CW-complex, 
and that $G$ acts freely except possibly that some vertices have 
isotropy subgroups that are of type $FP(R)$ (resp.~$FP_n(R)$).  
Suppose also that 
$X= \bigcup_{m\in \nn} X(m)$, where $X(m)\subseteq X(m+1)\subseteq 
\cdots\subseteq X$
is an ascending sequence of $G$-subcomplexes, each of which contains
only finitely many orbits of cells.  In this case $G$ is $FP(R)$ 
(resp.~$FP_n(R)$) if and only if for all $i$ (resp.~for all $i<n$) 
the sequence $\overline{H}_i(X(m);R)$ 
of reduced homology groups is essentially trivial.  
\end{theorem}

\begin{proposition} \label{propbieri}
If a group $Q$ is a retract of a group $G$, and $G$ is of type 
$FP(R)$ (resp.~type $FP_n(R)$) then so is $Q$.  
\end{proposition} 

\begin{proof} 
A group $\Gamma$ is $FP(R)$ if and only if $\Gamma$ has finite 
cohomological dimension, $n$, over $R$ and $\Gamma$ is type $FP_n(R)$.  
Since $Q$ is isomorphic to a subgroup of $G$, the cohomological 
dimension of $G$ is an upper bound for the cohomological dimension 
for $Q$.  Thus the claim for type $FP(R)$ will follow from the claim 
for type $FP_n(R)$ for all $n\in \nn$.  

According to~\cite[Theorem~1.3]{bieri}, a group $\Gamma$ is $FP_n(R)$
if and only if the following condition holds: for any directed system
$M_*$ of $R\Gamma$-modules with $\dirlim M_*=0$, one has that $\dirlim
H^k(\Gamma;M_*)=0$ for all $k\leq n$.  Now suppose that $M_*$ is a
directed system of $RQ$-modules such that $\dirlim M_*=0$, and that we
have maps $Q\rightarrowtail G\twoheadrightarrow Q$ whose composite is
the identity.  The surjection $G\twoheadrightarrow Q$ enables us to 
view $M_*$ as a directed system of $RG$-modules, and the direct limit 
of this system is still zero, since its underlying abelian group is 
unchanged.  To see that for each $k\leq n$, $\dirlim H^k(Q;M_*)=0$, 
note that the identity map of this group factors through 
$\dirlim H^k(G;M_*)$, which is zero because $G$ is type $FP_n(R)$.  
\end{proof}

\section{Classifying spaces for right-angled Artin groups}

Recall that a flag complex is a simplicial complex such that every finite 
mutually adjacent set of vertices spans a simplex.  
The right-angled Artin group $A_L$ associated to a flag complex $L$
is the group given by a presentation with generators the vertex set 
$L^0$ of $L$, subject only to the relations that the ends of each 
edge commute.  
\[A_L=\langle v\in L^0\,\,\,:\,\,\, [v,w]=1 
\,\,\hbox{for all $\{v,w\}\in L^1$}\,\rangle\] 
For example, if $L$ is an $n$-simplex then $A_L$ is 
free abelian of rank $n+1$, while if $L$ is 0-dimensional, $A_L$ 
is a free group of rank $|L^0|$.  

For a vertex $v\in L^0$, let $\ttt_v$ be a copy of the circle
$\rr/\zz$, and give $\ttt_v$ a cellular structure with one 0-cell
corresponding to the image of $\zz$ and one 1-cell.  For a simplex
$\sigma=(v_0,\ldots,v_n)$, let $\ttt_\sigma$ be the direct product
$\ttt_\sigma = \prod_{i=0}^n \ttt_{v_i}$.  If $\tau$ is a face of
$\sigma$ then $\ttt_\tau$ can be identified with a subcomplex of
$\ttt_\sigma$:
\[\ttt_\tau \cong \prod_{v\in \tau} \ttt_v \times \prod_{v\in \sigma - \tau} 
\{v\}.\]

Now let $\Delta$ be a simplex with the same vertex set as $L$, and 
view $L$ as a subcomplex of $\Delta$.  For each simplex $\sigma$ of 
$L$, we can view $\ttt_\sigma$ as a subtorus of $\ttt_\Delta$.  Define
the Salvetti complex $\ttt_L$ to be the union of these tori:  
\[\ttt_L= \bigcup_{\sigma\in L} \ttt_\sigma \subseteq \ttt_\Delta.\] 
The given cell structure on each $\ttt_v$ induces a cell structure
on $\ttt_L$, with one 0-cell and for each $n>0$, 
one $n$-cell for each $(n-1)$-simplex of $L$. 
Furthermore, the $n$-cell corresponding to the $(n-1)$-simplex $\sigma$
is naturally a copy of the $n$-cube $I^n$, in such a way that the 
$n$-torus $\ttt_\sigma$ is obtained by identifying opposite faces of $I^n$.  
The 2-skeleton of $\ttt_L$ is the presentation 2-complex for the defining
presentation of $A_L$, and $\ttt_L$ is an Eilenberg-Mac~Lane space 
$K(A_L,1)$.  
If we identify the circle $\ttt_v$ with the group 
$\rr/\zz$, so that the zero cell is the group identity, 
then we obtain a map $l_\Delta:\ttt_\Delta\rightarrow \rr/\zz$ defined by 
\[(t_1,\dots,t_n)\mapsto t_1+\cdots +t_n \quad\hbox{for all 
$(t_1,\ldots,t_n)\in \ttt_\Delta$}.\]
This map restricts to $\ttt_L\subseteq \ttt_\Delta$ to define 
a map $l_L:\ttt_L\rightarrow \rr/\zz$.  The induced map  
on fundamental groups is the homomorphism $l_*:A_L\rightarrow \zz$ that 
sends each of the standard generators for $A_L$ to $1\in \zz$, and 
the kernel of this homomorphism is by definition the Bestvina-Brady
group $BB_L$.  For each $v\in L^0$, the map $l_L$ restricts to a 
homeomorphism $l_v:\ttt_v\rightarrow \rr/\zz$.  

The link of the point 0 in $\rr/\zz$ is a copy of the 0-sphere, 
corresponding to travelling in the positive and negative directions
along $\rr$.   Similarly, the link of the 0-cell in $\ttt_v$ is a 
copy of the 0-sphere, which we shall denote by $\sss(v)=\{v^+,v^-\}$.  
Under the isomorphism of links induced by the homeomorphism 
$l_v:\ttt_v\rightarrow \rr/\zz$, the point $v^+$ (resp.~$v^-$) 
corresponds to the positive 
direction (resp.~negative direction).  The link of the vertex in 
$\ttt_\Delta$ is the $(|L^0|-1)$-sphere $\sss(\Delta)$, equal to the 
join $\sss(\Delta)= \bigast_{v\in \Delta} \sss(v)$.  
For each simplex $\sigma$ of $L$, 
there is an inclusion $\sss(\sigma)\subseteq \sss(\Delta)$, and the link 
$\sss(L)$ of the unique vertex of $\ttt_L$ is equal to the union of these 
images: 
\[\sss(L) = \bigcup_{\sigma\in L} \sss(\sigma) \subseteq \sss(\Delta).\] 

\section{A Morse function on $X_L$} 

Let $X=X_L$ denote the universal cover of $\ttt_L$, and let
$f=f_L:X_L\rightarrow \rr$ be the map of universal covers induced by
$l:\ttt_L\rightarrow \rr/\zz$.  Each lift of each cell of $\ttt_L$ is
an embedded cube in $X$, and so $X$ is a cubical complex, in which the
link of each vertex is isomorphic to $\sss(L)$.  Since $\sss(L)$ is a flag
complex (Proposition~\ref{propflag}), it follows that the path metric
on $X$ defined by using the standard Euclidean metric on each cube to
measure lengths of PL paths is CAT(0).  This metric gives one way to 
prove that $\ttt_L$ is a $K(A_L,1)$~\cite{brihae}.  

The map $f:X_L\rightarrow \rr$ has the following properties: its
restriction to each $n$-cube (identified with $[0,1]^n\subseteq
\rr^n$) is an affine map; the image of each vertex is in $\zz$; the
image of each $n$-cube is an interval of length~$n$.  Any map from a
cubical complex to $\rr$ having these properties will be called a
\emph{Morse function}; the definition can be made more general, but this 
will suffice for our purposes.  For a Morse function $f:X\rightarrow \rr$ and
$v$ a vertex of $X$, the ascending or $\upw$-link $\Lk^\upw_X(v)$ is
the subcomplex of $\Lk_X(v)$ coming from the cubes $C$ containing $v$
for which $f:C\rightarrow \rr$ attains its minimum at $v$; the
descending or $\downw$-link $\Lk^\downw_X(v)$ is defined similarly with
`minimum' replaced by `maximum'.  The Morse function $f:X_L\rightarrow
\rr$ has the property that for each vertex $v$, both $\Lk^\upw_X(v)$
and $\Lk^\downw_X(v)$ are naturally identified with $L$, being
the full subcomplexes of $\sss(L)$ spanned by $\{v^+;v\in L^0\}$ and
$\{v^-;v\in L^0\}$ respectively.  We will use the following theorem 
concerning Morse functions on cubical complexes~\cite{BB}.  

\begin{theorem} \label{thmmorse} 
Let $g:Y\rightarrow \rr$ be a Morse function on a cubical complex, 
let $a<b<c<d\in \rr$, and define $Y_{[b,c]}=g^{-1}([b,c])$ and
similarly $Y_{[a,d]}=g^{-1}([a,d])$.  Up to homotopy $Y_{[a,d]}$ 
is obtained from $Y_{[b,c]}$ by coning off a copy of $\Lk\downw_Y(v)$ 
for each vertex $v$ with $g(v)\in (c,d]$ and coning off a copy of 
$\Lk^\upw_Y(v)$ for each vertex $v$ with $g(v)\in [a,b)$.  
\end{theorem} 

Let $\tilde l:X_L/BB_L\rightarrow \rr$ be the map induced by
$f:X\rightarrow \rr$, or equivalently the pullback of the maps
$l:\ttt_L\rightarrow \rr/\zz$ and $\rr\rightarrow \rr/\zz$.  Note that
the map $\tilde l$ gives a bijection between the vertex set of
$X_L/BB_L$ and $\zz$.

\section{$L$ and $\sss(L)$} 

\begin{proposition} 
For a simplicial complex $L$ without isolated vertices, the 
following properties are equivalent: 
\begin{itemize} 
\item Every 1-simplex is contained in a 2-simplex and every vertex 
link is connected; 

\item The topological realization $|L|$ has no local cut point.  
\end{itemize} 
\end{proposition}

\begin{proof} 
For $n\geq 2$ the $n$-simplex has no local cut points; it follows 
that the only possible local cut points are either vertices of $L$ 
or points in an edge of $L$ that is not contained in any triangle. 
A non-isolated vertex is a local cut point if and only if its link
is not connected, and the midpoint of an isolated edge is always 
a local cut point.  
\end{proof} 

\begin{definition}\label{defnnlcp} 
Say that $L$ has nlcp if it satisfies the equivalent conditions 
listed above.  
\end{definition} 

\begin{lemma} \label{lemnlpc}
For any flag complex $L$ without isolated vertices, there is a 
flag complex $M=M(L)$ with $L\subseteq M$, and having the 
following properties.  
\begin{enumerate} 
\item $M$ has nlcp; 

\item $L$ is a full subcomplex of $M$; 

\item The inclusion $L\subseteq M$ is a homotopy equivalence; 

\item Provided that $L$ is at least 2-dimensional, $M$ has the same
  dimension as $L$.   
\end{enumerate} 

The construction $L\mapsto M(L)$ can be chosen to be functorial 
for maps $L_1\rightarrow L_2$ that do not collapse any simplex.  
\end{lemma} 

\begin{proof} 
Any triangulation of $|L|\times [0,1]$ for which $|L|\times\{0\}$ 
is realized as a full subcomplex isomorphic to $L$ will have 
properties (1)--(3).  For property~(4) to hold as well, we instead  
take $M(L)$ to be a suitable triangulation of the mapping cylinder 
of the inclusion of $L^1$ in $L$.  It is convenient to 
define $M(L)$ as a full subcomplex
of a triangulation $N(L)$ of $|L|\times [0,1]$, so that 
\[|M(L)|= |L|\times\{0\}\cup |L^1|\times [0,1]  \subseteq 
|N(L)|=|L|\times  [0,1].\]  
A construction for $N(L)$ appears implicitly in a well-known 
proof of excision for singular homology~\cite[theorem~2.20]{hatcher}. 
The vertices of $N(L)$ are the vertices of $L$, together with the 
simplices of $L$, distinguishing between the vertex $v$ and the 
0-simplex $\{v\}$.  For $-1\leq k\leq n$, a set 
$\{v_0,\ldots, v_k,\sigma_{k+1}\ldots,\sigma_n\}$ 
of $k+1$ vertices and $n-k$ simplices of $L$ spans a simplex 
of $N(L)$ if and only if the following conditions hold: 
$v_0,\ldots,v_k$ span a simplex of $L$; the simplices 
$\sigma_{k+1},\ldots,\sigma_n$ are totally ordered by 
inclusion; each $v_i$ is a vertex of each $\sigma_j$.  
There is a homeomorphism from $|N(L)|$ to $|L|\times [0,1]$ 
that sends $v$ to $(v,0)$ and sends $\sigma$ to $(\hat\sigma,1)$, 
where $\hat\sigma$ denotes the barycentre in $|L|$ of the 
the simplex $\sigma$.   
It is easily checked that this construction has properties 
(1)--(3), and is functorial for simplicial maps that do not 
collapse any simplex.  Furthermore if $M(L)$ is defined to 
be the full subcomplex of $N(L)$ with vertex set the vertices 
of $L$ together with the 0-~and 1-simplices of $L$, then 
$M(L)$ is a triangulation of $|L|\times \{0\}\cup |L^1|\times [0,1] 
\subseteq |L|\times [0,1]$, 
and $M(L)$ has properties (1)--(4).  
\end{proof} 

\begin{remark} Note that the set of embeddings $L\subseteq M$ 
having properties (1)--(4) is \emph{directed}, in the sense 
that if $L\subseteq M_1$ and $L\subseteq M_2$ have these 
properties then so does $L\subseteq M_1\cup_LM_2$.  
\end{remark} 

\begin{proposition}\label{propflag}
$L$ is a flag complex if and only if $\sss(L)$ is a flag complex.  
$L\subseteq M$ is a full subcomplex if and only if $\sss(L)\subseteq \sss(M)$
is a full subcomplex.  
\end{proposition} 

\begin{proof} 
See~\cite[lemma~5.8]{BB} for the first statement.  The second
statement is proved similarly.  
\end{proof}

\begin{proposition}\label{propretr}
For any $L$, $L$ is a retract of $\sss(L)$.  
\end{proposition} 

\begin{proof} 
For any function $g:L^0\rightarrow \{+,-\}$, the map 
$v\mapsto v^{g(v)}$ induces a simplicial map $L\rightarrow \sss(L)$ 
which is a pre-inverse to the projection map $\pi(v^\epsilon)= v$.  
\end{proof} 

\begin{proposition}\label{propcon}  
$\sss(L)$ is connected if and only if $L$ is connected and not equal to 
a single point. 
\end{proposition} 

\begin{proof} 
Since $L$ is a retract of $\sss(L)$, $\sss(L)$ cannot be connected in the case 
when $L$ is not.  In the case when $L$ is a single point, $\sss(L)$ is two 
points.  Conversely, suppose that $v^\epsilon$ and $w^{\epsilon'}$ are 
vertices of $\sss(L)$ with $\epsilon,\epsilon'=\pm$.  Given any path 
$v=v_0\ldots,v_l=w$ in $L$ with $l>0$, the path $v^\epsilon,v_1^+,v_2^+,
\ldots,v_{l-1}^+,w^{\epsilon'}$ is a path from $v^\epsilon$ to 
$w^{\epsilon'}$.  
\end{proof} 

\begin{proposition}\label{propsimpconn} 
If $L$ has nlcp, then for any vertex $v$ of $L$, 
$\sss(\Star_L(v))$ is simply-connected. 
\end{proposition} 

\begin{proof} 
The fact that $L$ has nlcp implies that $\Lk_L(v)$ is connected and
not a single point.  By Proposition~\ref{propcon}, it follows that
$\sss(\Lk_L(v))$ is connected.  Since $\sss(\Star_L(v))$ can be identified
with the join $\{v^+,v^-\}*\sss(\Lk_L(v))$, the claim follows, since the
join of a non-empty space and a connected space is simply-connected.
\end{proof} 

\begin{proposition} \label{propfundgp} 
The natural map $\pi_1(\sss(L))\rightarrow \pi_1(L)$ is always surjective. 
If $L$ has nlcp, then this map is an isomorphism. 
\end{proposition} 

\begin{proof} 
The first statement is immediate from Proposition~\ref{propretr}.  For
the second statement, it suffices to show that every edge loop in
$\sss(L)$ based at a positive vertex $v_0^+$ is based homotopic to an
edge loop contained in the positive copy of $L$.  Let
$v_0^+=v_0^{\epsilon_0},v_1^{\epsilon_1},\ldots,v_{l-1}^{\epsilon_{l-1}},
v_l^{\epsilon_l}=v_0^+$ be such an edge loop in $\sss(L)$.  If
$\epsilon_i=+$ for each $i$ with $0<i<l$ then this loop already lies
in the positive copy of $L$.  Otherwise, pick some $i$ so that
$\epsilon_i=-$.  The subpath $v_{i-1}^{\epsilon_{i-1}},
v_i^{\epsilon_i},v_{i+1}^{\epsilon_{i+1}}$ is contained in
$\sss(\Star_L(v_i))$, which is simply-connected by
Proposition~\ref{propsimpconn}.  Hence this path is homotopic relative
to its end points to the path $v_{i-1}^{\epsilon_{i-1}},
v_i^+,v_{i+1}^{\epsilon_{i+1}}$, and the original loop is based
homotopic to a based loop containing fewer negative vertices.
\end{proof} 

\section{Coverings of $L$ and $\sss(L)$} 

\begin{proposition} 
Let $L$ be a connected flag complex with universal covering 
$\pi:\widetilde{L}\rightarrow L$.  The following diagram is a
pullback square.  
\[\begin{array}{ccc}
\sss(\widetilde{L})&\mapright{\sss(\pi)}&\sss(L)\\
\mapdown{}&&\mapdown{}\\
\widetilde{L}&\mapright{\pi}&L
\end{array}\] 
\end{proposition} 

\begin{proof} 
The vertex set of $\sss(\widetilde{L})$ is easily seen to be the 
pullback of the other three vertex sets.  Furthermore, a set 
of $n+1$ vertices of $\sss(\widetilde{L})$ spans an $n$-simplex if 
and only if its images in both $\widetilde{L}$ and $\sss(L)$ span
$n$-simplices.  
\end{proof} 

\begin{corollary}\label{covcor} 
For $L$ connected and not a single point, 
$\sss(\pi):\sss(\widetilde{L})\rightarrow \sss(L)$ is 
a covering of $\sss(L)$ with fundamental group equal to the 
kernel of $\pi:\pi_1(\sss(L))\rightarrow \pi_1(L)$.  
\end{corollary} 

\begin{proof} 
$\sss(\pi)$ is a covering map, since it is the pullback of the
  covering map $\pi:\widetilde{L}\rightarrow L$, and since the map
  $\sss(\widetilde{L})\rightarrow L$ factors through $\widetilde{L}$,
  the fundamental group of $\sss(\widetilde{L})$ has trivial image in
  $\pi_1(L)$.  Conversely, if $\gamma$ is an edge loop in $\sss(L)$
  whose image in $L$ is null-homotopic, then the image of $\gamma$
  lifts to an edge loop $\gamma'$ in $\widetilde{L}$.  The pair
  $(\gamma,\gamma')$ defines an edge loop $\gamma''$ in
  $\sss(\widetilde{L})$ with $\sss(\pi)(\gamma'')=\gamma$.
\end{proof}

When the complex $\sss(L)$ arises as the vertex link in $\ttt_L$, 
its vertices are naturally partitioned into `opposite pairs' 
$v^+,v^-$.  This partition is not, in general, determined by 
the isomorphism type of the simplicial complex $K=\sss(L)$.  For 
example, if $L=\sss(\Delta)$ for $\Delta$ an $n$-simplex and 
$K=\sss(L)$, then $K$ is isomorphic to the join of $(n+1)$ sets 
of four vertices, and $K$ itself gives no clue as to how these
four element sets pair up.  However, we show that the opposite pairs
in $\sss(\widetilde{L})$ are determined completely by the covering 
map $\sss(\pi):\sss(\widetilde{L})\rightarrow \sss(L)$ and the opposite pairs 
in $\sss(L)$.

\begin{proposition} \label{oppprop} 
For $L$ connected and not a single point  
the vertices $v$, $w$ of $\sss(\widetilde{L})$ form an opposite 
pair if and only if the following conditions hold: 
\begin{itemize} 
\item there is an edge path in $\sss(\widetilde{L})$ of length two from 
$v$ to $w$;

\item $\sss(\pi)(v)$ and $\sss(\pi)(w)$ form an opposite pair in $\sss(L)$.  
\end{itemize} 
\end{proposition} 

\begin{proof} 
If $v$ and $w$ form an opposite pair in $\sss(\widetilde{L})$ then so do
their images in $\sss(L)$.  Also for any adjacent vertices $x$, $y$ of
$\widetilde{L}$, $x^+,y^+,x^-$ is an edge path of length two in
$\sss(\widetilde{L})$ joining the opposite pair $x^+,x^-$.  

Now suppose that there are edge paths of length two in
$\sss(\widetilde{L})$ joining $v$ to both $w$ and $w'$, so that
$\sss(\pi)(w)=\sss(\pi)(w')$ forms an opposite pair with $\sss(\pi)(v)$.  By
concatenating these we obtain a length four path $w,u,v,u',w'$ from
$w$ to $w'$, which maps to a loop in $\sss(L)$.  Since $w$, $v$ and $w'$
all map to the same vertex of $L$, the image of this loop in $L$ is
null-homotopic.  By Corollary~\ref{covcor}, the original path in
$\sss(\widetilde{L})$ must be a closed loop and hence $w=w'$.
\end{proof}

\section{Local properties of CAT(0)-cube complexes} 

Note that in a cubical complex, the link of a vertex is naturally 
isomorphic to the sphere of radius $1/4$ around that vertex.  In 
a CAT(0) cubical complex, we can say more.  

\begin{theorem}\label{thmdefretr}
For any CAT(0) cubical complex $X$ and any vertex $v$ of $X$, 
$\link_X(v)$ is a strong deformation retract of $X-\{v\}$. 
\end{theorem} 

\begin{proof}
For any $x\in X-\{v\}$, there is a unique $y\in \link_X(v)$ such that
the geodesic arcs $[v,y]$ and $[v,x]$ are nested: if $d(x,v)=1/4$
then $y=x$; if $d(x,v)>1/4$ then $y$ is the unique point of 
$\link_X(v)\cap [x,v]$, and if $d(x,v)<1/4$ this follows 
from~\cite[I.7.16]{brihae}.  Define $r:X-\{v\}\rightarrow \link_X(v)$ 
by $r(x)=y$, for $y$ defined as above.  If $i$ denotes the inclusion of 
the link as the sphere of radius $1/4$, 
$i:\link_X(v)\rightarrow X-\{v\}$, then $r\circ i$ is the identity map 
$\id_{\link_X(v)}$, and one can define a homotopy from $\id_X$ to $i\circ r$ 
by moving from $x$ to $r(x)$ at constant speed along the geodesic 
arc $[x,r(x)]$.  
\end{proof} 

\begin{corollary} Let $Z$ be a locally CAT(0) cube complex with 
all vertex links connected, and let $W$ be a set of vertices of 
$Z$.  Then $Z-W$ is connected and for any $w\in W$, the inclusion
map $i:\link_Z(w)\rightarrow Z-W$ is $\pi_1$-injective.  
\end{corollary} 

\begin{proof} $Z$ is path-connected by definition, and connectivity
of vertex links implies that any PL path between points of $Z-W$ 
may be homotoped to avoid $W$.  Now 
let $p:X\rightarrow Z$ be the universal covering map, 
so that $X$ is a CAT(0) cubical complex, and let $V=p^{-1}(W)$, a 
set of vertices of $X$, and fix some $v$ with $p(v)=w$.  Since 
$i:\link_Z(w)\rightarrow Z$ is homotopic to the constant map sending 
every point of $\link_Z(w)$ to $w$, it follows that $p:\link_X(v)
\rightarrow \link_Z(w)$ 
is a homeomorphism.  Now $p:X-V\rightarrow Z-W$ is a covering
of connected spaces, and so the induced map 
$p_*:\pi_1(X-V)\rightarrow \pi_1(Z-W)$ is injective.  Hence 
it suffices to show that the map $i':\link_X(v)\rightarrow X-V$ is injective on
fundamental groups.  This follows from Theorem~\ref{thmdefretr}, 
since the composite $\link_X(v)\rightarrow X-V \rightarrow X-\{v\}$ 
is split injective on fundamental groups.  
\end{proof} 

\begin{theorem} \label{thmcatocov} 
Let $X$ be a locally CAT(0) cube complex with universal covering 
$p:\widetilde{X}\rightarrow X$, suppose that all vertex links in 
$X$ are connected, and let $V$ be a set of vertices of $X$.  
There is a CAT(0) cubical complex $\overline{X}$ and a regular
branched covering map $c:\overline{X}\rightarrow X$, with 
branching locus contained in $V$ and such that 
\[\Lk_{\overline{X}}(w)= \begin{cases} 
\Lk_X(c(w)) \mbox{ if }c(w)\notin V \\
\widetilde{\Lk_X(c(w))} \mbox{ if }c(w)\in V.\end{cases} \] 
Moreover, the covering map $c$ factors through $\widetilde{X}$, 
in the sense that $c=p\circ b$ for some regular branched covering
map $b:\overline{X}\rightarrow \widetilde{X}$.  
\end{theorem}

\begin{proof} 
We shall define $\overline{X}$ by adding vertices to the universal cover 
$Y$ of $X-V$.  Let $c:Y\rightarrow X-V$ denote the covering map.  
Since the map $\pi_1(X-V)\rightarrow \pi_1(X)$ is 
surjective, it follows that $c:Y\rightarrow X-V$ factors through 
$\widetilde{X}-p^{-1}(V)$, so we define 
$b:Y\rightarrow \widetilde{X}-p^{-1}(V)$ so that $c=p\circ b$.  
The connected components of the 
inverse images of cubes of $\widetilde{X}$ split $Y$ as a 
disjoint union of open cubes.  This gives $Y$ the structure 
of a cubical complex with some missing vertices.  
Let $W'$ be the set of those ends of edges of $Y$ that do not 
have a vertex incident on them.  Generate an equivalence relation
on $W'$ by the relation $e\sim e'$ if there is a corner of a 
square of $Y$ that contains the edge-ends $e$ and $e'$.  Now 
define $W$ to be $W'/\sim$, the set of equivalence classes, 
and let $\overline{X}=Y\coprod W$.  For $n>0$, each $n$-cube 
of $\overline{X}$ is an $n$-cube of $Y$.  The vertex set of 
$\overline{X}$ consists of $W$ together with the vertices 
of $Y$.  A vertex $w\in W$ is incident on an edge-end $e$ 
if and only $w$ is the equivalence class defined by $e$.  
For $w\in W$, $b(w)$ is defined as follows: pick some 
edge-end $e$ incident on $w$, and define $b(w)$ 
to be the vertex of $\widetilde{X}$ incident on the edge-end $b(e)$, 
and define $c(w)=p(b(w))$.  The claimed structure for 
the vertex links in $\overline{X}$ is now easily verified.  
Since $Y$ is simply-connected and the links of vertices 
in $W$ are connected, it follows that $\overline{X}$ is 
simply-connected, and hence that $\overline{X}$ is CAT(0) 
as claimed.  
\end{proof} 

There is an alternative description of $\overline{X}$ as the metric 
completion of $Y$ in the path metric induced by $c:Y\rightarrow X-V$, 
but we prefer to emphasize the combinatorial nature of
$\overline{X}$.  

\section{The construction of $X^{(S)}_L$} 

\begin{theorem}\label{thmexess}
For any connected finite $L$ and any $S\subseteq \zz$, 
the locally CAT(0) cubical complex $X_L/BB_L$ admits a regular branched 
cover $c:X_L^{(S)}\rightarrow X_L/BB_L$ with the following 
properties: 
\begin{itemize} 
\item $c$ factors through $X_L$: there is a regular branched cover 
$b:X_L^{(S)}\rightarrow X_L$ whose composite with the natural map 
$X_L\rightarrow X_L/BB_L$ is equal to $c$; 

\item The function $f^{(S)}:= f\circ b=\tilde l\circ c$ is a Morse function on
  $X^{(S)}$; 

\item The links of vertices of $X^{(S)}$ are either $\sss(L)$ or 
$\sss(\widetilde{L})$, depending on their height: 
\[\Lk_{X^{(S)}}(v)=\begin{cases} 
\sss(L) \mbox{ if }f^{(S)}(v)\in S\\
\sss(\widetilde{L}) \mbox{ if }f^{(S)}(v)\notin S
\end{cases} \]

\item 
The ascending and descending links of vertices of $X^{(S)}$ are either $L$ 
or $\widetilde{L}$ depending on their height: 
\[\Lk\downw_{X^{(S)}}(v)=\Lk\upw_{X^{(S)}}(v)
=\begin{cases} 
L \mbox{ if }f^{(S)}(v)\in S\\
\widetilde{L} \mbox{ if }f^{(S)}(v)\notin S
\end{cases} \]
\end{itemize} 

The group $G_L(S)$ is by definition the group of deck transformations 
of the branched covering $c:X_L^{(S)}\rightarrow X_L/BB_L$.  
\end{theorem} 

\begin{proof} 
In the case when $L$ has nlcp, Proposition~\ref{propfundgp} shows that 
$\sss(\widetilde{L})$ is the universal cover of $\sss(L)$; in this 
case $X_L^{(S)}$ and $b$ can be constructed by applying 
Theorem~\ref{thmcatocov} to the locally CAT(0) cube complex 
$X_L/BB_L$, with $V$ equal to the vertices of $X_L/BB_L$ 
whose height lies in $\zz-S$.  Since $X_L/BB_L$ has exactly one 
vertex of each height we note that $G_L(S)$ acts transitively 
on the vertices of $X_L^{(S)}$ of each height.  

In the general case, we realize $X_L^{(S)}$ as a subcomplex of 
$X_M^{(S)}$, where $M=M(L)$ is as described in Lemma~\ref{lemnlpc}.  
We have $X_L/A_L$ embedded as a subcomplex of $X_M/A_M$.  Define 
$Y$ to be one of the connected components of the inverse image 
of $X_L/A_L \subseteq X_M/A_M$ under the composite map 
$X_M^{(S)}\rightarrow X_M\rightarrow X_M/A_M$.  Since the 
group $G_M(S)$ acts transitively on the vertices of $X_M^{(S)}$ 
of a given height and $Y$ contains vertices of all heights, 
any other connected component will be isomorphic to $Y$.  
Since $L$ is a full subcomplex of $M$, $\sss(L)$ is a full subcomplex
of $\sss(M)$, and the inverse image of $\sss(L)\subseteq \sss(M)$ under
the universal covering map $\sss(\widetilde{M})\rightarrow \sss(M)$ 
is $\sss(\widetilde{L})$.  It follows that vertices of $Y$ of 
height in $S$ have link $\sss(L)$ and vertices of $Y$ of height
not in $S$ have link $\sss(\widetilde{L})$ as required.  

Since each vertex link in $Y$ is a full subcomplex of the vertex
link in $X_M^{(S)}$, it follows that $Y$ is a convex subspace of 
$X_M^{(S)}$, and so $Y$ is simply-connected and hence is itself 
CAT(0).  By construction, we have branched covering maps 
$b:Y\rightarrow X_L$ and $c:Y\rightarrow X_L/BB_L\subseteq X_M/BB_M$.  
It remains to show that the second of these coverings 
is regular.  For this, it suffices to show that whenever $g\in G_M(S)$ 
is such that $gY\cap Y\neq \emptyset$, then $gY=Y$.  This is because
the action of $g$ preserves the `type' of cubes.  Suppose that $y,z$ 
are vertices of $Y$ and that $y\in gY\cap Y$, and let $e_1,\ldots, 
e_l$ be an edge path in $Y$ from $y$ to $z$.  Equivalently, 
$e_1,\ldots,e_l$ is an edge path in $X_M^{(S)}$ such that 
each $e_i$ has image contained in the subspace $X_L/A_L\subseteq 
X_M/A_M$.  The edge path $g^{-1}e_1,\ldots,g^{-1}e_n$ starts at 
$y'\in Y$ and has image contained in the subspace $X_L/A_L$ 
of $X_M/A_M$, and hence its end point $z'$ must lie in $Y$.  
From this we see that the edge path $e_1,\ldots, e_l$ is contained 
in $Y\cap gY$, so that $z\in gY$ and hence that 
$Y=gY$.  This establishes that the group 
\[G_L(S)=\{g\in G_M(S): gY\cap Y\neq \emptyset\} = 
\{g\in G_M(S): gY = Y\}\]
acts regularly on the branched covering $c:Y\rightarrow X_L/BB_L$.  

It remains to show that this construction does not depend on the
choice of $M$.  Suppose that $L\subseteq M_1$ and $L\subseteq M_2$ are
two choices of embedding as in Lemma~\ref{lemnlpc}, yielding $Y_1$ and
$Y_2$ with groups $G_1$ and $G_2$ of deck transformations.  If we
define $M_3:= M_1\cup_LM_2$, then $L\subseteq M_3$ is also an
embedding with the properties of Lemma~\ref{lemnlpc}, and we may
construct $Y_3\subseteq X_{M_3}^{(S)}$ with an action of $G_3\leq
G_{M_3}(S)$.  The naturality properties lead to maps $Y_1\rightarrow
Y_3$ and $Y_2\rightarrow Y_3$ which are local isomorphisms and hence
(since all three are simply-connected) isomorphisms.  This in turn
yields group isomorphisms $G_1\rightarrow G_3$ and $G_2\rightarrow
G_3$ with respect to which the given isomorphisms $Y_i\rightarrow Y_3$
are equivariant.
\end{proof} 

\begin{corollary} \label{corstabs}
The cellular action of the group $G_L(S)$ on $X_L^{(S)}$ is free,
except that each vertex whose height is not in $S$ is stabilized 
by a subgroup isomorphic to $\pi_1(L)$. 
\end{corollary} 

\begin{proof} 
Immediate from the construction of $X_L^{(S)}$.
\end{proof} 

Fix some connected $L$, and for $t\in \rr$, let $X_t=X_{L,t}$ 
denote the level set $f_L^{-1}(t)\subseteq X_L$.  

\begin{corollary}\label{corregcov}
If $t\notin \zz$, or $t\in S$, then 
then $c:X^{(S)}_t\rightarrow X_t/BB_L$ is a regular cover 
without branching, and induces an isomorphism 
$X^{(S)}_t/G_L(S)\cong X_t/BB_L$.  
\end{corollary} 

\begin{proof} 
No branch point of $c:X^{(S)}\rightarrow X/BB_L$ has height equal to
$t$.  
\end{proof} 

\section{Connectivity of level sets in $X^{(S)}_L$} 

As at the end of the previous section, we fix a choice of connected $L$, 
and for $t\in \rr$ and $S\subseteq \zz$ we write $X^{(S)}_t$ for the 
level set ${f^{(S)}}^{-1}(t)\subseteq X^{(S)}_L$.   Furthermore, we 
write $X^{(S)}_{[a,b]}$ for ${f^{(S)}}^{-1}([a,b])\subseteq X^{(S)}_L$.

\begin{proposition} 
\begin{enumerate}
\item Each $X^{(S)}_t$ is connected.  

\item If both $L$ and $\widetilde{L}$ are $R$-acyclic, then $X^{(S)}_t$ is 
$R$-acyclic.  

\item If $L$ and $\widetilde{L}$ are both $m$-$R$-acyclic then so is
  $X^{(S)}_t$.  
\end{enumerate}
\end{proposition} 

\begin{proof} 
(1) Let $[a,b]$ be any interval in $\rr$ containing $t$.  
The ascending and descending links for each vertex of $X^{(S)}$ are 
connected, from which it follows that each inclusion $X^{(S)}_t\rightarrow 
X^{(S)}_{[a,b]}$ induces a bijection on the set of path components.  
Since $X^{(S)}$ is contractible, it follows that each level set must
be connected, as claimed.  

(2)~and~(3) Under the hypotheses the ascending and descending links
for each vertex in $X^{(S)}$ are either $R$-acyclic, or have 
reduced $R$-homology that vanishes in degree at most $m$.  
Theorem~\ref{thmmorse} implies that 
for each interval $[a,b]$ with $a\leq t\leq b$, 
the relative homology groups $H_i(X^{(S)}_{[a,b]},X^{(S)}_t;R)$ are zero 
either for all $i$ or for $i\leq m+1$.  It follows that the 
map $H_i(X^{(S)}_t;R) \rightarrow H_i(X^{(S)};R)$ is an 
isomorphism, either for all $i$ or for 
$i\leq m$.  Since $X^{(S)}$ is contractible the result follows.  
\end{proof} 

This already allows us to establish part of the main theorem.  

\begin{corollary}\label{cormorse} 
If $L$ and $\widetilde{L}$ are $R$-acyclic 
(resp.~$(n-1)$-$R$-acyclic) 
then for each $S\subseteq \zz$, $G_L(S)$ is 
$FH(R)$ (resp.~$FH_n(R)$).  
\end{corollary} 

\begin{proof} 
$G_L(S)$ acts freely cocompactly on the level set $X^{(S)}_t$, 
which is $R$-acyclic (resp.~$(n-1)$-$R$-acyclic) by the proposition.  
\end{proof} 

\begin{proposition} 
For any $t$, $X^{(\emptyset)}_t$ is simply-connected.  For $t=0$, 
$X^{(\{0\})}_0$ is simply-connected. 
\end{proposition} 

\begin{proof} 
In each of these cases, every vertex not of height $t$ has ascending
and descending links isomorphic to $\widetilde{L}$, which is 
simply-connected.  From this it follows that the inclusion of 
$X_t$ into $X_{[a,b]}$ induces an isomorphism of fundamental 
groups for each interval $[a,b]$ containing $t$.  Since $X$ is 
simply-connected, the claim follows. 
\end{proof} 

\begin{corollary}\label{corgzero}
For any $t\notin \zz$, $G_L(\emptyset)$ is isomorphic to
$\pi_1(X_t/BB_L)$.  For $t=0$, $G_L(\{0\})$ is isomorphic 
to $\pi_1(X_0/BB_L)$.  
\end{corollary} 

\begin{proof} 
Define $S$ by $S=\emptyset$ for $t\notin \zz$ or $S=\{0\}$ for 
$t=0$.  In each case, $X^{(S)}_t$ is the universal covering 
of $X_t/BB_L$, with $G_L(S)$ as its group of deck transformations.  
\end{proof} 

\begin{corollary} 
$G_L(\emptyset)\cong BB_{\widetilde {L}}\semi \pi_1(L)$.   
\end{corollary} 

\begin{proof} 
It suffices to identify $BB_{\widetilde{L}}\semi\pi_1(L)$ with the 
fundamental group of $X_t/BB_L$ for $t\notin \zz$.  We start by
constructing a branched covering of $X_L$ by $X_{\widetilde {L}}$.  
The action of $\pi_1(L)$ on $\widetilde{L}$ by deck transformations
induces an action on $\ttt_{\widetilde{L}}$ which 
fixes the unique vertex and is free at other points.  Furthermore, 
there is a natural identification $\ttt_{\widetilde{L}}/\pi_1(L) 
= \ttt_L$.  This exhibits a branched regular covering map 
$\ttt_{\widetilde{L}} \rightarrow \ttt_L$, with branching 
only at the vertex and with $\pi_1(L)$ as its group of deck 
transformations.  The induced map of universal covering 
spaces is a branched covering $X_{\widetilde {L}}\rightarrow 
X_L$, which is regular and equivariant for the actions of the 
groups $A_{\widetilde{L}}\semi\pi_1(L)=\pi_1(\ttt_{\widetilde{L}})\semi 
\pi_1(L)$ and $A_L$ respectively.  The composite of the branched
covering map and the Morse function $f:X_L\rightarrow \rr$ induces a
Morse function on $\widetilde{f}:X_{\widetilde{L}}\rightarrow\rr$, 
in which each ascending and descending link is isomorphic to 
$\widetilde{L}$.  Furthermore, the subgroup of
$A_{\widetilde{L}}\semi\pi_1(L)$ that preserves 
level sets of this Morse function is
$BB_{\widetilde{L}}\semi\pi_1(L)$.  

Since all ascending and descending links in $X_{\widetilde{L}}$ are 
1-connected, we see that for any $t$, the level set
$X_{\widetilde{L},t}=\widetilde{f}^{-1}(t)$ is simply-connected.  For 
$t\notin \zz$, we see that $X_{\widetilde{L},t}\rightarrow X_t/BB_L$ 
is a universal covering map, with $BB_{\widetilde{L}}\semi\pi_1(L)$ 
as its group of deck transformations.  Hence we get isomorphisms 
$BB_{\widetilde{L}}\semi\pi_1(L) \cong \pi_1(X_t/BB_L)\cong
G_L(\emptyset)$ by comparing the two descriptions of $\pi_1(X_t/BB_L)$.  
\end{proof} 

\begin{corollary} For each $S\subseteq \zz$, $G_L(S)$ is finitely
  generated. 
\end{corollary} 

\begin{proof} 
$G_L(\emptyset)$ is the fundamental group of the compact polyhedral
  complex $X_t/BB_L$ for any $t\notin \zz$, and so is finitely
  generated.  Each $G_L(S)$ can be viewed as the group of deck
  transformations of a regular covering of $X_t/BB_L$, and hence as a
  quotient of $G_L(\emptyset)$.  
\end{proof} 

\section{$G_L(S)$ as a fundamental group} 

In this section we realize $G_L(S)$ as either a fundamental 
group (if $L$ has nlcp) or as the image of a map of fundamental
groups (in general).  This clarifies $G_L$ as a functor of $S$
and allows us to exhibit isomorphisms between $G_L(S)$ and $G_L(T)$ 
in some cases.  

Recall that $X_L/BB_L$ is equipped with a height function 
$l:X_L/BB_L\rightarrow \rr$.  The height function $l$ 
establishes a bijection between the vertex set of $X_L/BB_L$ 
and the integers.  For $S\subseteq \zz$ we define $V(S)=V_L(S)$ 
to be the set of vertices whose height is \emph{not} in $S$: 
\[V(S)=\{v\,\, \hbox{a vertex of}\,\, X_L/BB_L \colon l(v)\notin S\}.\] 

\begin{proposition} 
If $L$ has nlcp, then $G_L(S)\cong \pi_1(X_L/BB_L-V(S))$.   
\end{proposition}

\begin{proof} 
Let $W(S)$ denote the set of vertices of $X_L^{(S)}$ that have height
not in $S$, so that $W(S)$ consists of the vertices of $X_L^{(S)}$ 
with link $\sss(\widetilde{L})$.  Note also that $W(S)$ is equal to the 
set of points of $X_L^{(S)}$ with non-trivial stabilizer in $G_L(S)$.  
Since $L$ has nlcp, $\sss(\widetilde{L})$ is 1-connected, from which it 
follows that $X_L^{(S)}-W(S)$ is simply-connected, and so $G_L(S)$ 
can be identified with the group of deck transformations of the 
universal covering of $X_L^{(S)}-W(S)$.  
\end{proof} 

To describe $G_L(S)$ in terms of fundamental groups when $L$ does not
have nlcp, we embed $L$ as a full subcomplex of $M=M(L)$ as in the 
statement of Lemma~\ref{lemnlpc}.  

\begin{proposition}
In general, $G_L(S)$ is isomorphic to the image of the map 
\[\pi_1(X_L/BB_L -V_L(S))\rightarrow \pi_1(X_M/BB_M-V_M(S)).\]
\end{proposition} 

\begin{proof} 
Recall that in the general case, we defined $X_L^{(S)}$ as a 1-connected 
subspace of $X_M^{(S)}$ for $M=M(L)$ as in Lemma~\ref{lemnlpc}, 
and we defined $G_L(S)$ as the subgroup of $G_M(S)$ that preserves 
this subspace.  As before, $G_L(S)$ acts freely on $X_L^{(S)}-W_L(S)$, 
and although this space need not be simply-connected, it admits a 
$G_L(S)$-equivariant map to the simply-connected space
$X_M^{(S)}-W_M(S)$.  Since $G_L(S)\leq G_M(S)$, this gives an 
identification of $G_L(S)$ with the claimed image.  
\end{proof} 

\begin{remark} \label{rkfunctor} 
For each $S\subseteq T\subseteq \zz$, there are inclusions of spaces 
$X_L/BB_L-V(S)\subseteq X_L/BB_L-V(T)$ and 
$X_M/BB_M-V(S)\subseteq X_M/BB_M-V(T)$.  The induced maps of
fundamental groups give a way to view $S\mapsto G_L(S)$ as a 
functor from subsets of $\zz$ with inclusion to groups and 
surjective homomorphisms.    
\end{remark}

An affine isometry of $\zz$ is by definition a map of the form 
$\phi(n)= m+n$ or $\phi(n)=m-n$ for some $m\in \zz$.  

\begin{corollary} 
If $\phi(S)=T$ for some affine isometry $\phi$, then $G_L(S)\cong
G_L(T)$.  
\end{corollary} 

\begin{proof} 
Extend $\phi$ to an affine isometry of $\rr$ given by the same 
formula, and as usual let $l:X_L/BB_L\rightarrow \rr$ be the 
height function on $X_L/BB_L$.  Define a new height function 
on $X_L/BB_L$ as $l'=\phi\circ l$, and note that $l(v)\in S$ 
if and only if $l'(v)\in T$.  
\end{proof}

Provided that $L$ is not simply-connected, we know of no cases when 
$G_L(S)$ is isomorphic to $G_L(T)$ except those that arise as from 
this corollary.  

\begin{corollary}\label{brancov} 
For any $S\subseteq T\subseteq \zz$, $X_L^{(S)}$ is a regular branched
cover of $X_L^{(T)}$, branched only at vertices of height in $T-S$.  
\end{corollary} 

\begin{proof} 
The viewpoint of Remark~\ref{rkfunctor} gives a map $X_L^{(S)}-W(S)
\rightarrow X_L^{(T)}-W(T)$ which is equivariant for the map 
$G_L(S)\twoheadrightarrow G_L(T)$.  (Here as before, we let $W(S)$ 
denote the vertices of $X_L^{(S)}$ that have height not in $S$.)  
If $K$ denotes the kernel of the map $G_L(S)\rightarrow G_L(T)$, 
then $K$ acts freely on $X_L^{(S)}-W(S)$ and one obtains an 
isomorphism $(X_L^{(S)}-W(S))/K\rightarrow X_L^{(T)}-W(T)$.  
Completing this map gives an isomorphism $X_L^{(S)}/K\cong
X_L^{(T)}$.  The vertices of $X_L^{(S)}$ that are fixed by
some non-trivial element of $K$ are precisely the vertices 
of height in $T-S$.  
\end{proof}

\section{Sheets in $X^{(S)}_L$} 

An $n$-flat in a CAT(0) complex is a subcomplex isometric to $\rr^n$.  
In~\cite{BB} an $n$-sheet in $X_L$ is defined to be an $n$-flat whose 
image in $\ttt_L$ is a single subtorus $\ttt_\sigma$ for some 
$(n-1)$-simplex $\sigma$ of $L$.  Since we do not have a cocompact 
group action on every $X^{(S)}_L$ this definition does not apply 
directly.  

An $n$-sheet in $X^{(\emptyset)}_L=X_{\widetilde{L}}$ can be defined to 
be an $n$-flat whose image in $\ttt_{\widetilde{L}}$ is a subtorus
$A_\sigma$ for some $(n-1)$-simplex of $\widetilde{L}$.  Now an 
$n$-sheet in $X^{(S)}_L$ can be defined to be the image of any 
$n$-sheet of $X^{(\emptyset)}_L$ under the branched covering map 
$X^{(\emptyset)}_L\rightarrow X^{(S)}_L$.  We prefer the following
equivalent definition: an $n$-sheet in $X^{(S)}$ is an $n$-flat $F$ 
with the property that for every vertex $v$ of $X^{(S)}$, the 
subcomplex $\Lk_F(v)\subseteq \Lk_{X^{(S)}}(v)$ is closed under 
taking opposite pairs.  Here we use Proposition~\ref{oppprop} to 
define opposite pairs in $\Lk_{X^{(S)}}(v)$.  

Under the branched covering map $X^{(S)}_L\rightarrow \ttt_L$, each 
$n$-sheet will map to an $n$-subtorus $\ttt_\sigma$, however for 
$S\neq \zz$ there are many $n$-flats with this property that are not 
$n$-sheets.  

For $v$ a vertex of $X_L^{(S)}$, the link $\Lk_X(v)$ is a copy 
of $\sss(L)$ if $f(v)\in S$ or $\sss(\widetilde{L})$ if $f(v)\notin S$.  
For each vertex $v$ with $f(v)\in S$ (resp.~with $f(v)\notin S$) 
and each simplex $\sigma$ of $L$ (resp.~of $\widetilde{L}$) there 
is a unique sheet $C(v,\sigma)$ containing $v$ and such that 
$\Lk_C(v)=\sss(\sigma)\subseteq \Lk_X(v)$.  Let $\Lambda$ denote 
$\rr^n$ with its standard tesselation by unit cubes, and let 
$\phi:\Lambda\rightarrow C(v,\sigma)$ be an isomorphism with 
the following properties: $\phi((0,0,\ldots,0))=v$, and 
$f(\phi(t_1,\ldots,t_n))=f(v)+\sum_{i=1}^n t_i$.  Now define 
the upward and downward parts of the sheet by 
\[C_\upw(v,\sigma) = \phi((\rr_{\geq 0})^n),\qquad 
C_\downw(v,\sigma) = \phi((\rr_{\leq 0})^n).\] 
Thus $C_\upw(v,\sigma)$ is the points $x$ of $C(v,\sigma)$ for
which the geodesic from $v$ to $x$ leaves $v$ in a point of 
$\Lk_\upw(v)\cap \sss(\sigma)$ and similarly for the downward 
part.  

If $f(v)\in S$ (resp.~$f(v)\notin S$), and 
$K$ is any subcomplex of $L$ (resp.~$\widetilde{L}$), define 
\[C(v,K)=\bigcup_{\sigma\in K}C(v,\sigma). \qquad
C_\upw(v,K)=\bigcup_{\sigma\in K}C_\upw(v,\sigma),\qquad 
C_\downw(v,K)=\bigcup_{\sigma\in K}C_\downw(v,\sigma).\]
For $f(v)>t$, the shadow $S_{v,K}$ is $X_t\cap C_\downw(v,K)$, 
and for $f(v)<t$ the shadow $S_{v,K}$ is $X_t\cap C_\upw(v,K)$.  

\begin{proposition} \label{propshadow}
For any $t< f(v)$, and any $K$, the retraction map 
$X-\{v\}\rightarrow \Lk_X(v)$ induces a homeomorphism from $S_{v,K}$
to $K\subseteq \Lk^\downw_X(v)\subseteq \Lk_X(v)$.  Similarly, 
if $t> f(v)$ then the same retraction map induces a 
homeomorphism from $S_{v,K}$ to $K\subseteq \Lk^\upw_X(v)\subseteq 
\Lk_X(v)$.  
\end{proposition} 

\begin{proposition} 
Suppose that $t\notin \zz$ and that either $f(v)\in S$ and $K=L$ 
or $f(v)\notin S$ and $K=\widetilde{L}$.  Let $t\in [a,b]$, and  
let $X_{[a,b]}$ denote $f^{-1}([a,b])$.  
\begin{itemize}
\item If $f(v)\in [a,b]$ then $S_{v,K}$ is null-homotopic in
  $X_{[a,b]}$; 

\item If $f(v)\notin [a,b]$ then $S_{v,K}$ is a retract of
  $X_{[a,b]}$.  
\end{itemize} 
\end{proposition} 

\begin{proof} 
We consider only the case $f(v)>t$, the contrary case being similar. 
In this case the retraction map $r_v:X-\{v\}\rightarrow \Lk_X(v)$
restricts to $S_{v,K}$ as a homeomorphism $h:S_{v,K}\rightarrow
\Lk^\downw_X(v)$.  Let $q:\Lk_X(v)\rightarrow \Lk^\downw_X(v)$ be 
the simplicial map sending each vertex to the downward member of 
its opposite pair.  If $f(v)\notin [a,b]$ then the composite 
\[h^{-1}\circ q\circ r_v:X_{[a,b]}\rightarrow S_{v,K}\subseteq 
X_{[a,b]}\] 
is the required retraction.  If on the other hand $f(v)\in [a,b]$, 
then by moving points of $S_{v,K}$ at constant speed along the 
geodesic joining them to $v$ we get a homotopy from the identity 
map of $S_{v,K}$ to the constant map with image $\{v\}$.  
\end{proof} 

\section{Brown's criterion applied to $X_L^{(S)}$} 
\label{secbrown}

To be able to apply Brown's criterion to decide whether $G_L(S)$ is 
type $FP(R)$ or type $FP_n(R)$, we will need to know something about
the finiteness properties of the cell stabilizers in $X_L^{(S)}$.  By 
Corollary~\ref{corstabs}, the only non-trivial cell stabilizers are 
the stabilizers of the vertices whose height is not in $S$.  The following 
proposition shows that each of the conditions stated to be equivalent
in Theorem~\ref{thmb} implies the required conditions on $\pi_1(L)$.  

\begin{proposition}\label{propbrownhyp} 
If $\widetilde{L}$ is $R$-acyclic, or if $G_L(\emptyset)$ is type
$FP(R)$, then $\pi_1(L)$ is type $FP(R)$. 

If $\widetilde{L}$ is $(n-1)$-R-acyclic 
or if $G_L(\emptyset)$ is 
type $FP_n(R)$, then $\pi_1(L)$ is type $FP_n(R)$.  
\end{proposition} 

\begin{proof} 
$\pi_1(L)$ acts freely on $\widetilde{L}$ by deck transformations. 
Hence if $\widetilde{L}$ is $R$-acyclic 
(resp.~$(n-1)$-$R$-acyclic) 
then $\pi_1(L)$ is by definition $FH(R)$
(resp.~$FH_n(R)$), and hence $\pi_1(L)$ is $FP(R)$ (resp.~$FP_n(R)$).  

Since $G_L(\emptyset)\cong BB_{\widetilde{L}}\semi \pi_1(L)$, we see
that $\pi_1(L)$ is a retract of $G_L(\emptyset)$.  By
Proposition~\ref{propbieri} it follows that $\pi_1(L)$ is $FP(R)$ 
(resp.~$FP_n(R)$) whenever $G_L(\emptyset)$ is.  
\end{proof} 

\begin{theorem}\label{thmfpgeneral} 
Suppose that $\pi_1(L)$ is type $FP(R)$ (resp.~type $FP_n(R)$).  
Then $G_L(S)$ is type $FP(R)$ (resp.~type $FP_n(R)$) if and 
only if the set of heights of vertices whose ascending links 
are not $R$-acyclic 
(resp.~are not $(n-1)$-$R$-acyclic) 
is finite.  
\end{theorem} 

\begin{proof} 
By Proposition~\ref{propbrownhyp}, the action of $G=G_L(S)$ on 
$X=X_L^{(S)}$ satisfies the hypotheses of Theorem~\ref{brownscrit}, 
our statement of Brown's criterion.  Now consider the filtration 
of $X$ by the $G$-subcomplexes 
$X(m)=X_{[-m-1/2,m+1/2]}=f^{-1}([-m+1/2,m+1/2])$ for $m\in \nn$.  $G$
acts cocompactly on each $X(m)$, and so $G$ is $FP(R)$ (resp.~$FP_n(R)$)
if and only if for each $i$ (resp.~for each $i<n$) the system
$\overline{H}_i(X(m);R)$ is essentially trivial.  

If the set of vertices whose ascending links are not $R$-acyclic
(resp.~are not $(n-1)$-$R$-acyclic) is finite, 
pick $m_0\in \nn$ so that $|f(v)|>m_0$ implies that $\Lk^\upw(v)\cong 
\Lk^\downw(v)$ is $R$-acyclic (resp.~$(n-1)$-$R$-acyclic).  Then for 
all $m'>m>m_0$ and for all $i$ (resp.~for all $i<n$)
$H_i(X(m'),X(m);R)=0$.  Hence in this case $G$ is $FP(R)$
(resp.~$FP_n(R)$) by Brown's criterion.  

Conversely, suppose that for each $m\in \nn$ there exists $m'>m+1$ and 
$v$ with $|f(v)|=m'$ such that $K=\Lk^\upw_X(v)$ is not
$(n-1)$-$R$-acyclic.  By Proposition~\ref{propshadow}, $S_{v,K}$ is 
a retract of $X(m)$ but is null-homotopic in $X(m')$.  Hence there 
exists $i$ (resp.~$i<n$) so that $\overline{H}_i(X(m);R)
\rightarrow \overline{H}_i(X(m');R)$ 
has non-trivial kernel, and so by Brown's criterion $G$ cannot be 
$FP(R)$ (resp.~$FP_n(R)$).  
\end{proof} 

\begin{corollary}\label{corfpr}
Suppose that $\pi_1(L)$ is type $FP(R)$ (resp.~type $FP_n(R)$).  

\begin{itemize} 
\item 
If $S$ is finite then 
$G_L(S)$ is $FP(R)$ (resp.~$FP_n(R)$) if and only if $\widetilde{L}$ 
is $R$-acyclic (resp.~$(n-1)$-$R$-acyclic).  

\item 
If $\zz-S$ is finite then $G_L(S)$ is $FP(R)$ (resp.~$FP_n(R)$) if and 
only if $L$ is $R$-acyclic (resp.~$(n-1)$-$R$-acyclic).  

\item
If $S$ and $\zz-S$ are both infinite, then $G_L(S)$ is $FP(R)$
(resp.~$FP_n(R)$) if and only if both $L$ and $\widetilde{L}$ are 
$R$-acyclic (resp.~$(n-1)$-$R$-acyclic).  
\end{itemize} 
\end{corollary} 

This corollary together with Corollary~\ref{cormorse} completes the
proof of Theorem~\ref{thmb}.  There are some cases with either $S$ or $\zz-S$ 
finite in which the above corollary shows that $G_L(S)$ is $FP(R)$ 
and yet Corollary~\ref{cormorse} does not apply.  For the cases 
when $S$ is finite we will show that $G_L(S)$ is in fact $FH(R)$ 
in Proposition~\ref{propsfinfhr}.

\section{Presentations for $G_L(S)$} 

In this section we study presentations for the groups $G_L(S)$.  
Since we want to realize the surjective group homomorphism
$G_L(S)\rightarrow G_L(T)$ by a map of presentations, it is 
convenient to fix a finite generating set, and for this reason
we focus on the case when $0\in S$.  

\begin{theorem}\label{thmpres} 
For any $S$ with $0\in S$ and any collection $\Gamma$ of directed
loops in $L$ whose normal closure generates $\pi_1(L)$, the 
presentation $P_L(S,\Gamma)$ is a presentation of the group
$G_L(S)$.  If $S\subseteq T$ then the pair $P_L(S,\Gamma)\subseteq 
P_L(T,\Gamma)$ realizes the surjection $G_L(S)\rightarrow G_L(T)$.
\end{theorem} 

\begin{remark} 
The case $S=\{0\}$ of this statement first appeared in~\cite{howie}
and was stated in~\cite{DL}.  A proof of the case $S=\zz$ without 
using the cube complex $\ttt_L$ was given in~\cite{DL}.  
\end{remark}

\begin{proof}  
We start with the case $S=\{0\}$.  By Corollary~\ref{corgzero},
$G_L(\{0\})$ may be viewed as the fundamental group of $X_0/BB_L$.  
The presentation $P_L(\{0\},\Gamma)$, which does not depend on 
$\Gamma$, arises from the 2-skeleton of the natural
polyhedral cell complex structure on $X_0/BB_L$.  Cells of the 
polyhedral structure on $X_0$ are intersections $C\cap X_0$, 
where $C$ is a cube of $X$.  The vertices of $X_0$ form a  
single $G$-orbit.  Edges meet $X_0$ only at their end points, 
and for each $n\geq 2$, each $A_L$-orbit of $n$-cubes of $X$ 
contributes $(n-1)$ distinct $G$-orbits of $(n-1)$-cells of 
$X_0$, corresponding to the distinct integer heights at which 
members of the orbit can appear.  Each $A$-orbit of squares 
contributes one $G$-orbit of edges, each $A$-orbit of 
3-cubes contributes two $G$-orbits of triangles, and each
$A$-orbit of 4-cubes contributes two $G$-orbits of tetrahedra
plus one $G$-orbit of octahedra.  
The 1-skeleton of $X_0/BB_L$ is a rose consisting of one vertex
and edges in bijective correspondence with the edges of $L$, 
or equivalently the squares of $\ttt_L$.  If $x$~and~$y$ are 
commuting Artin generators, then the two directions along the 
corresponding edge represent $xy^{-1}$ and $yx^{-1}$.  The
2-skeleton is formed by attaching two triangles for each triangle 
of $L$; if $(a,b,c)$ is the edge loop around a directed triangle 
of $L$ then the attaching maps are $abc$ and $a^{-1}b^{-1}c^{-1}$, 
giving the claimed presentation.  
\end{proof} 

Before proving the general case, we record a useful lemma.  
For each $S$ with $0\in S$, the edges of $X_0^{(S)}/G_L(S)$ are 
naturally bijective with the edges of $L$.  Hence given any 
initial vertex $w$ in $X_0^{(S)}$, any word in the directed 
edges of $L$ describes a unique directed edge path starting at $w$.  

\begin{lemma}\label{leminshadow} 
Let $(a_1,\ldots,a_l)$ be any directed edge path in $L$.  
For any $S$ with $0\in S$, for $n\in \zz-\{0\}$ and 
any initial vertex in $X^{(S)}_0$, the 
directed edge path defined by the word $a_1^na_2^n\cdots a_l^n$ is
contained in the shadow of a vertex $v$ of height~$n$, and is equal 
to the image of the path $(a_1,\ldots,a_l)$ under the 
isomorphism $L\rightarrow S_{v,L}$.  
\end{lemma} 

\begin{proof} 
First consider the case when $l=1$, and suppose that $n>0$.  If $a$ is
the directed edge in $L$ from vertex $x$ to vertex $y$, then for any
vertex $\scw\in X^{(S)}_0 \subseteq X=X_L^{(S)}$ there is an $n\times
n$ square with vertices $\scw,\scn,\sce,\scs$ of heights $0,n,0,-n$
respectively such that the height~0 diagonal from $\scw$ to $\sce$
represents the word $a^n$, while the boundary paths from $\scw$ to
$\scn$ and from $\scs$ to $\sce$ represent the word $x^n$ and the
boundary paths from $\sce$ to $\scn$ and from $\scs$ to $\scw$
represent the word $y^n$.  The square is of course contained in the
2-sheet through $\scw$ defined by the edge $a$, and is uniquely
determined by $a$, $n$ and any one of its vertices.  The edge paths
$a^n$ and $a^{-n}$ are each the image of a single edge under the
shadow map from either $\scn$ or $\scs$.  This proves the claim 
whenever $l=1$.  

For the general case, fix $n>0$ and fix $i$ with $1\leq
i<l$, let $a=a_i$ and $b=a_{i+1}$, where $a$ is the directed edge from 
vertex $x$ to vertex $y$ and $b$ is the directed edge from vertex 
$y$ to vertex $z$.  As before there are $n\times n$ squares with 
vertices $\scw,\scn,\sce,\scs$ and $\scw',\scn',\sce',\scs'$ such 
that the respective height~0 diagonals represent the words $a^n$ and 
$b^n$.  If $\sce=\scw'$, so that the concatenated path from $\scw$ to 
$\sce'$ represents $a^nb^n$, then we see that $\scn=\scn'$, since both
are vertices of $X$ that can be reached from $\sce=\scw'$ by moving 
along the path $y^n$.  Similarly, if $\sce'=\scw$ so that the 
concatenated path from $\sce$ to $\scw'$ represents $a^{-n}b^{-n}$, 
then $\scs=\scs'$.  Thus the word $a_1^n\cdots a_l^n$ is the image  
of $(a_1,\ldots,a_l)$ under the shadow map from $\scn$ and 
$a_1^{-n}\cdots a_l^{-n}$ is the image of the same path under the 
shadow map from $\scs$.  Since $\scn$ has height $n$ and $\scs$ 
has height $-n$, the claim follows.  
\end{proof} 

\begin{proof} (Theorem~\ref{thmpres}, general case) 
For general $S$ with $0\in S$, we use Morse theory to compare
$G_L(\{0\})=\pi_1(X^{(S)}_0/G_L(S))$ with $G_L(S)=
\pi_1(X^{(S)}/G_L(S))$.  Since $X_0^{(S)}$ is a regular covering space
of $X_0/BB_L$ (Corollary~\ref{corregcov}), we may view words in the
directed edges of $L$ as defining edge paths in $X^{(S)}_0$, and a
word represents the identity in $G_L(S)$ if and only if the
corresponding path in $X^{(S)}_0$ is a closed loop.  We build
$X^{(S)}$ as the direct limit of the $G_L(S)$-subspaces
$X(m)=X^{(S)}_{[-m-1/2,m+1/2]}$ for $m\in \nn$.  Up to homotopy, $X(m+1)$ is
obtained from $X(m)$ by attaching cones to the shadows of vertices of
height $m+1$ in $X_m\subseteq X(m)$ and attaching cones to the shadows
of vertices of height $-(m+1)$ in $X_{-m}\subseteq X(m)$.  Up to 
homotopy, this is the same as attaching cones to the shadows of 
these vertices in $X_0$.  If vertices
of height $m$ (resp.~height $-m$) have link $\sss(\widetilde{L})$, then 
since $\widetilde{L}$ is simply connected, the fundamental group is 
left unchanged.  If the vertices of a given height have link $\sss(L)$, 
then attaching a cone to the shadow in $X_0$ of the vertices of that 
height has the same effect on the fundamental group as attaching 
discs to the images under the shadow map of all edge loops in 
$\Gamma$.  Hence we see that the fundamental group of $X(m)/G_L(S)$ 
has presentation $P_L(S\cap [-m,m],\Gamma)$, and hence that 
$P_L(S,\Gamma)$ is a presentation for $G_L(S)=\pi_1(X/G_L(S))$.  
\end{proof}  

\begin{lemma}\label{lemmarels}
Suppose that $0\in S$ and that $\gamma=a_1,\ldots,a_l$ is a
homotopically non-trivial edge loop in $L$.  The relation 
$a_1^na_2^n\cdots a_l^n=1$ holds in $G_L(S)$ if and only if 
$n\in S$.  
\end{lemma} 

\begin{proof} 
Since $0\in S$ the case when $n=0$ is trivial.  If $n\neq 0$, 
the relation holds if and only if any edge path in $X_0^{(S)}$
described by the word $a_1^na_2^n\cdots a_l^n$ is a closed loop.  
By Lemma~\ref{leminshadow}, 
any such edge path is the image, under the shadow map from a 
vertex $v$ of height $n$, of the edge path $(a_1,\ldots, a_l)$ in 
either the ascending or descending link of $v$ (depending whether 
$n$ is less than or greater than 0).  If $f(v)\in S$ then this 
edge path is a closed loop in a complex isomorphic to $L$.  If on 
the other hand $f(v)\notin S$ then this edge path lies in a 
complex isomorphic to $\widetilde{L}$ and is not closed.  The 
claim follows since the shadow maps are homeomorphisms.  
\end{proof} 

\begin{corollary} 
$G_L(S)$ is finitely presentable if and only if $S$ is finite. 
\end{corollary} 

\begin{proof}  
If $S$ is finite, then for any finite set of edge loops 
$\Gamma$ that normally generates $\pi_1(L)$, $P_L(S,\Gamma)$ is a 
finite presentation of $G_L(S)$.  For the converse, fix a choice of 
$\Gamma$, and consider the presentation $P_L(S,\Gamma)$.  If $G_L(S)$ 
has a finite presentation then these finitely many relators are 
consequences of finitely many of the relators in the presentation 
$P_L(S,\Gamma)$.  It follows that when $G_L(S)$ is finitely
presentable, then there must be a finite subset $S'\subseteq S$ 
so that $P_L(S',\Gamma)$ also presents $G_L(S)$.  By
Lemma~\ref{lemmarels} this cannot happen unless $S'=S$.  
\end{proof} 

If $T$ is non-empty but does not contain $0$, then $G_L(T)$ is
isomorphic to $G_L(S)$ where $S$ is a translate of $T$ that contains
zero.  Thus we obtain a presentation for $G_L(T)$ whenever $T$ is 
non-empty, albeit with a non-canonical choice of generating set. We 
finish this section by giving a finite presentation for
$G_L(\emptyset)$.  

Since $\widetilde{L}$ is simply-connected, the group 
$G_{\widetilde{L}}(S)$ is isomorphic to $BB_{\widetilde{L}}$ for each
$S\subseteq \zz$.  Hence Theorem~\ref{thmpres} gives a 
presentation for $BB_{\widetilde{L}}$: the generators are the
directed edges of $\widetilde{L}$, subject only to the edge 
relations and the pairs of triangle relations coming from 
the triangles of $\widetilde{L}$.  Of course, this presentation
$P_{\widetilde{L}}(\{0\},\emptyset)$ is infinite whenever $\pi_1(L)$ 
is, but the generators and relators are both permuted freely by 
$\pi_1(L)$ and lie in finitely many orbits.  From this we obtain 
a finite presentation of the group
$G_L(\emptyset)=BB_{\widetilde{L}}\semi\pi_1(L)$.  Choose a finite 
presentation for $\pi_1(L)$, and choose a finite fundamental 
domain $K$ in $\widetilde{L}$ for the action of $\pi_1(L)$, 
i.e., a finite subcomplex $K$ that contains at least one 
simplex from each $\pi_1(L)$-orbit.  The generators in the presentation 
are the generators of the given presentation for $\pi_1(L)$, together
with the directed edges of $K$, subject to four types of relation: 
the finite set of relators coming from the given presentation for 
$\pi_1(L)$; the edge relators for each of the finitely many edges 
in $K$; the pair of triangle relators for each of the finitely 
many triangles in $K$; conjugacy relations: whenever $a$ is a 
directed edge of $K$ and $g\in \pi_1(L)-\{1\}$ is an element
such that the image of $a$ under $g$ is another edge $b$ of $K$, 
the relation $gag^{-1}=b$.  

The map $G_L(\emptyset)\rightarrow G_L(\{0\})$ is easily described in
terms of this presentation: the elements of $\pi_1(L)$ are sent to the 
identity and the directed edges of $K\subseteq \widetilde{L}$ are sent to the 
corresponding directed edges of $L$.  

\section{A set-valued invariant}

\begin{definition} 
For a group $G$ and a finite sequence $\ug=(g_1,\ldots, g_l)$ of elements of 
$G$, define a set $\cR(G,\ug)$ with $\{0\}\subseteq \cR(G,\ug)\subseteq \zz$ 
by 
\[ \cR(G,\ug) = \{n\in \zz\colon g_1^ng_2^n\cdots g_l^n = 1\}.\]
\end{definition} 

\begin{proposition} \label{proprgprops} 
Let $\ug=(g_1,\dots,g_l)$ be a sequence of elements of a group $G$.  
\begin{enumerate} 
\item
If $G\leq H$, then $\cR(G,\ug)=\cR(H,\ug)$;    
\item 
If $G$ is finitely presented then the set $\cR(G,\ug)$ is 
recursively enumerable; 
\item 
For any fixed isomorphism type of 
countable group $G$, the invariant $\cR(G,\ug)$ 
can take at most countably many distinct values.  
\end{enumerate} 
\end{proposition} 

\begin{proof} 
If $G\leq H$ then any equation between elements of $G$ holds in $G$ 
if and only if it holds in $H$.

Suppose we are given a finite presentation of the 
group $G$, consisting of a set $x_1,\ldots, x_d$ of generators, 
together with a set $r_1,\ldots,r_m$ of relators, given as words 
in $x_i^{\pm 1}$.  Suppose also that we are given words $w_1,\ldots,w_l$ 
in $x_i^{\pm 1}$ such that $g_i=w_i$.  We work in the free group 
freely generated by $x_1,\ldots, x_d$, where there is an easy algorithm 
to replace a word by a reduced word.  Now for each fixed $N>0$,  
construct the list $\cL(N)$ consisting of the reductions of words 
obtained as a product of at most $N$ words of 
the form $wr_i^{\pm 1}w^{-1}$, where $w$ denotes a word in 
$x_1^{\pm 1},\ldots, x_d^{\pm 1}$ of length at most $N$.  For each 
$n$ with $|n|\leq N$, check whether the reduction of the word 
$w_1^n\cdots w_l^n$ is in $\cL$; whenever this happens, output $n$.  
Keep repeating this programme for larger and larger values of $N$.  

If $G$ is countable then $G$ contains only countably many finite 
sequences $\ug$ of elements, so the function $\ug\mapsto \cR(G,\ug)$ 
can take only countably many values.  If $\phi:G\rightarrow H$ is 
an isomorphism and $\phi(\ug)$ denotes the sequence $(\phi(g_1),\ldots
,\phi(g_l))$ of elements of $H$, then clearly $\cR(H,\phi(\ug))=
\cR(G,\ug)$.  
\end{proof} 

\begin{lemma} \label{lemmarggls}
Suppose that $\gamma=(a_1,\ldots,a_l)$ is a
  homotopically non-trivial edge loop in $L$, and let $\ua$ be the 
sequence $(a_1,\ldots,a_l)$ of elements of $G_L(S)$ for any $S$ 
containing zero.  In this case, $\cR(G_L(S),\ua)=S$.  
\end{lemma}

\begin{proof}  
This is immediate from Lemma~\ref{lemmarels}. 
\end{proof} 

\begin{corollary} 
If $\pi_1(L)$ is non-trivial then there are uncountably many
(in fact $2^{\aleph_0}$) isomorphism types of group $G_L(S)$. 
\end{corollary} 

\begin{proof} 
Follows from part~3 of Proposition~\ref{proprgprops} and
Lemma~\ref{lemmarggls}. 
\end{proof}

\begin{corollary} 
If $\pi_1(L)$ is non-trivial then 
$G_L(S)$ is a subgroup of a finitely presented group if
  and only if $S$ is recursively enumerable.  
\end{corollary}  

\begin{proof} 
First note that $G_L(\emptyset)$ is finitely presented.  Next note
that $S\subseteq \zz$ is recursively enumerable if and only if $S+m$
is, and $G_L(S)\cong G_L(S+m)$.  By these observations it suffices to
consider the case when $0\in S$.  If $G_L(S)$ is a subgroup of a
finitely presented group, then by Proposition~\ref{proprgprops},
$S=\cR(G_L(S),\ua)$ must be recursively enumerable.  For the converse,
if $S$ is recursively enumerable and $\Gamma$ is any finite set of
edge loops that normally generate $\pi_1(L)$ then $P_L(S,\Gamma)$ is a
recursive presentation for $G_L(S)$.  Higman's embedding theorem
implies that $G_L(S)$ embeds in a finitely presented
group~\cite{higman}.
\end{proof}  

This completes the proof of Theorem~\ref{thma}.  

\section{A classifying space for $G_L(S)$} 

The group $G_L(S)$ does not act freely on $X_L^{(S)}$ except 
when $S=\zz$ or $\pi_1(L)$ is trivial.  However, since only 
vertices have non-trivial stabilizers it is comparatively 
easy to construct a classifying space for $G_L(S)$.  A 
suitable small neighbourhood in $X_L^{(S)}$ of a singular vertex is 
$\pi_1(L)$-equivariantly isomorphic to the mapping cylinder
of the map $\sss(\widetilde{L})\rightarrow *$, where $*$ denotes
a point with trivial $\pi_1(L)$-action.  If at each such vertex
we replace this mapping cylinder by the mapping cylinder of 
the map $\sss(\widetilde{L})\rightarrow E\pi_1(L)$ then we obtain 
a contractible complex with a free $G_L(S)$ action.  Note that 
the map $\sss(\widetilde{L})\rightarrow E\pi_1(L)$ factors through 
$\widetilde{L}$.  There are other variations on this construction
that have useful corollaries.  

\begin{proposition}\label{propsfinfhr}
If $S$ is finite and $\widetilde{L}$ is $R$-acyclic
(resp.~$(n-1)$-$R$-acyclic) then $G_L(S)$ is $FH(R)$
(resp.~$FH_n(R)$).  
\end{proposition} 

\begin{proof} 
To simplify the notation we concentrate on the case when
$\widetilde{L}$ is $R$-acyclic.  
As in the proof of Theorem~\ref{thmfpgeneral}, let
$X(m)=f^{-1}([-m-1/2,m+1/2]) \subseteq X_L^{(S)}$.  Provided that $m$
is sufficiently large that $S\subseteq [-m,m]$, the inclusion of 
$X(m)$ into $X_L^{(S)}$ is an $R$-homology isomorphism, and so $X(m)$ 
is $R$-acyclic, and $G=G_L(S)$ acts cocompactly but not freely.  We 
modify $X(m)$ to obtain a new $R$-acyclic space $Y$ with a free
$G$-action as follows.  
Remove from $X(m)$ an open ball of radius $1/4$ around each 
vertex $v$ with non-trivial stabilizer.  The closure of each such ball 
is $G_v$-equivariantly isomorphic to the mapping cylinder 
$M(\sss(\widetilde{L})\rightarrow *)$, where $G_v$ denotes the stabilizer 
of $v$.  Replace the removed ball by a copy of the mapping cylinder  
$M(\sss(\widetilde{L})\rightarrow \widetilde{L})$ to obtain $Y$.  
Since $\widetilde{L}$ is $R$-acyclic this replacement does not change the 
$R$-homology, and hence $Y$ is $R$-acyclic.  By construction, $G$ acts 
freely on $Y$, and the quotient $Y/G$ is obtained from the finite 
complex $X(m)/G$ by replacing finitely many copies of
$M(\sss(L)\rightarrow *)$ by copies of $M(\sss(L)\rightarrow L)$; in
particular $Y/G$ is finite and hence $G$ is $FH(R)$.  

In the case when $\widetilde{L}$ is only $(n-1)$-$R$-acyclic, the 
argument is similar, except that now we only know that the mapping 
cylinders $M(\sss(\widetilde{L})\rightarrow *)$ and 
$M(\sss(\widetilde{L})\rightarrow \widetilde{L})$ have the same
$R$-homology in degrees less than or equal to $(n-1)$.  Hence in 
this case the free $G$-complex $Y$ is only $(n-1)$-$R$-acyclic, 
and we deduce only that $G$ is $FH_n(R)$.  
\end{proof} 

If $L$ is $R$-acyclic and $\zz-S$ is finite but $\widetilde{L}$ is 
not $R$-acyclic, then Corollary~\ref{corfpr} implies that $G_L(S)$ 
is $FP(R)$.  In this case we do not know whether $G_L(S)$ is 
$FH(R)$, even if we assume that $\pi_1(L)$ is $FH(R)$.  Let $E$ be 
an $R$-acyclic complex with a free cocompact action of $\pi_1(L)$.  
Given a $\pi_1(L)$-equivariant map $\sss(\widetilde{L})\rightarrow E$, 
the technique used above would show that $G$ is $FH(R)$.  

\begin{theorem} 
If $L$ is $R$-acyclic then for any $S$, there is an isomorphism of 
ordinary group homology
\[H_*(G_L(S);R)\cong H_*(G_L(\zz);R)\oplus\bigoplus_{\zz-S} \overline
H_*(\pi_1(L);R).\]
This isomorphism is natural for inclusions $S\subseteq T$.  
\end{theorem} 

\begin{remark} The ordinary homology and cohomology of $BB_L\cong G_L(\zz)$ 
was computed for all $L$ in~\cite{LS}.  In the case when $L$ is
$R$-acyclic, $H_i(G_L(\zz);R)$ is a free $R$-module isomorphic to the 
$(i-1)$-cycles in the augmented chain complex for $L$.  
\end{remark} 

\begin{proof} 
Let $Y$ be the free $G=G_L(S)$-complex in which $X_L^{(S)}$ is 
desingularized by replacing a ball around each singular vertex 
with the mapping cylinder of the $\pi_1(L)$-equivariant map 
$\sss(\widetilde{L})\rightarrow \widetilde{L}$.  There is a 
$G$-equivariant map $Y\rightarrow X$ given by collapsing 
the copies of $\widetilde{L}$ to the vertices that they replaced. 
Let $E$ be the universal covering space of any model for the classifying space
for $\pi_1(L)$, fix a classifying map $\widetilde{L}\rightarrow E$, 
and let $Z$ be built from $X_L^{(S)}$ by replacing a ball around 
each singular vertex 
with the concatenation of the two mapping cylinders for the 
following $\pi_1(L)$-equivariant maps: 
\[\sss(\widetilde{L})\rightarrow \widetilde{L}\rightarrow E.\]  
By construction, $Z$ is contractible, and $G$ acts freely on $Z$ 
so that $Z/G$ is a classifying space for $G$.  

Now $X_L^{(S)}/G\cong X_L/BB_L=X_L/G_L(\zz)$, and $Y/G$ is obtained
from $X_L^{(S)}/G$ by replacing a ball around each vertex with height 
in $\zz-S$ by the mapping cylinder of $\sss(L)\rightarrow L$.  Since 
$L$ is $R$-acyclic, this implies that the map $Y/G\rightarrow X/G$ 
(given by collapsing each copy of $L$ down to the vertex it replaces) 
is an $R$-homology isomorphism.  Now $Z/G$ is obtained from $Y/G$ by 
attaching a copy of the mapping cylinder of $L\rightarrow E/\pi_1(L)$
to the copy of $L$ that replaced the vertex at each height in
$\zz-S$.  Since $L$ is $R$-acyclic, the Mayer-Vietoris sequence 
expressing $Z/G$ as the union of $Y/G$ and $\coprod_{\zz-S}
M(L\rightarrow E/\pi_1(L))$ gives the claimed isomorphism.  
\end{proof} 

\begin{corollary}\label{coracyc} 
In the case when $L$ and $\widetilde L$ are both $R$-acyclic, for 
each $S\subseteq T$ the natural map $H_*(G_L(S);R)\rightarrow H_*(G_L(T);R)$ 
is an isomorphism.  
\end{corollary} 

\begin{proof} 
If $\widetilde{L}$ is $R$-acyclic, then $H_*(L;R)$ is isomorphic to
the $R$-homology of the group $\pi_1(L)$.  Hence when $L$ is also 
$R$-acyclic, the group $\pi_1(L)$ is $R$-acyclic.  
\end{proof} 

\begin{theorem} 
If $L$ and $\widetilde L$ are both $R$-acyclic, then for each 
$S\subseteq T$, the kernel of the map $G_L(S)\rightarrow G_L(T)$ 
is $R$-acyclic. 
\end{theorem} 

\begin{proof} 
Let $K$ be the kernel of the map $G_L(S)\rightarrow G_L(T)$.  
By Corollary~\ref{brancov}, $K$ acts on $X_L^{(S)}$ freely except
that each vertex of height in $T-S$ has stabilizer isomorphic to 
$\pi_1(L)$.  As above, a suitable small neighbourhood of each 
such vertex can be identified with the mapping cylinder of the map 
$\sss(\widetilde{L})\rightarrow *$.  By replacing each such neighbourhood
by a copy of the mapping cylinder for the map 
$\sss(\widetilde{L})\rightarrow \widetilde{L}$ we obtain a space 
$X$ on which $K$ acts freely.  Since $\widetilde{L}$ is $R$-acyclic, 
$X$ has the same $R$-homology as $X_L^{(S)}$ and so $X$ is $R$-acyclic. 
Hence $X$ can be used to compute group cohomology with coefficients in 
any $RK$-module.  By Corollary~\ref{brancov}, $X/K$ is homeomorphic to 
the space $Y$ constructed as follows.  In $X_L^{(T)}$, each vertex of 
height in $T$ has a neighbourhood that can be identified with the 
mapping cylinder of $\sss(L)\rightarrow *$.  The space $Y$ is obtained by 
removing such a neighbourhood for each vertex of height in $T-S$, 
and replacing it with a copy of the mapping cylinder for 
$\sss(L)\rightarrow L$.  Since $L$ is $R$-acyclic, the $R$-homology and 
$R$-cohomology of $Y$ is isomorphic to that of the contractible space
$X_L^{(T)}$, which proves the claim.  
\end{proof}

\begin{corollary} 
Let $R$ be a principal ideal domain and suppose that $L$ and
$\widetilde{L}$ are both $R$-acyclic.  For any $S\subseteq T$ and any
$RG_L(T)$-module $M$, the natural map
\[H^*(G_L(T);M)\rightarrow H^*(G_L(S);M)\] 
is an isomorphism.  
\end{corollary} 

\begin{proof} 
Let $K$ be the kernel of the homomorphism $G_L(S)\twoheadrightarrow 
G_L(T)$.  Since the $G_L(S)$-action on $M$ is defined in terms of the
quotient map, every element of $K$ acts as the identity on $M$.  Hence 
$H^0(K;M)=M$ and by the universal coefficient theorem, $H^i(K;M)=0$ 
for $i>0$.  Now consider the Lyndon-Hochschild-Serre spectral sequence
for the group extension $K\rightarrowtail G_L(S)\twoheadrightarrow 
G_L(T)$ with coefficients in $M$.  In this spectral sequence 
$E_2^{i,j}=H^i(G_L(T);H^j(K;M))$ and the terms $E_\infty^{i,j}$ form 
a filtration of $H^{i+j}(G_L(S);M)$.  Since $E_2^{i,j}=0$ for $j\neq 0$, 
the spectral sequence collapses at the $E_2$-page.  The claimed isomorphism
follows.
\end{proof}

\section{Compactly-supported and $\ell^2$-cohomology} 

In~\cite{do}, Davis and Okun computed the cohomology with compact
supports and the $\ell^2$-cohomology of $X_t$ for each 
finite flag complex $L$ and any non-integer $t$.  (Throughout 
this section we will write $X$ instead of $X_L$ to simplify 
the notation.)  Their methods extend almost without change to 
a similar computation for $X_t^{(S)}$ for any $S\subseteq \zz$, 
$L$ and $t\notin \zz$.  In particular, this computes the
$\ell^2$-cohomology of the group $G_L(S)$ whenever both $L$ and 
$\widetilde{L}$ are $\qq$-acyclic, and computes (a filtration of) 
$H^*(G_L(S);RG_L(S))$ whenever $L$ and $\widetilde{L}$ are
$R$-acyclic.  We indicate briefly how to extend their methods to 
our situation and state the results obtained in this way.  

Davis and Okun decompose $X$ as a union of $BB_L$-orbits of sheets, 
and rely on the fact that the stabilizer of each sheet is a free 
abelian group.  Following their approach, we fix a non-integer $t$, 
and define a \emph{level sheet} or $t$-\emph{level sheet} to be an 
intersection $C_t=C\cap X^{(S)}_t$, for some sheet $C$ of $X^{(S)}$.  
If $C$ is an $n$-dimensional sheet in $X^{(S)}$, then the image 
of $C$ in $X^{(S)}/G_L(S)\cong X/BB_L$ is isomorphic to
$\rr^n/\zz^{n-1}$, and the image of $C_t$ in $X_t/BB_L$ 
is isomorphic to the $(n-1)$-torus 
$\rr^{n-1}/\zz^{n-1}$.  It follows that the stabilizer of the 
level sheet $C_t$ is isomorphic to $\zz^{n-1}$.  Although sheets 
in the same $G_L(S)$-orbit can intersect, they intersect in at 
most a single vertex, so the $t$-level sheets in a $G_L(S)$-orbit 
do not intersect.  Since each $n$-simplex $\sigma$ of $L$ gives 
rise to a $G_L(S)$-orbit of $(n+1)$-sheets, we see that there 
is a natural bijection between the $n$-simplices of $L$ and the 
$G_L(S)$-orbits of $n$-dimensional level sheets in~$X^{(S)}_t$.  

Put a polyhedral cell structure on the level set $X^{(S)}_t$ 
in which the cells are the intersections of cubes of $X^{(S)}$ with 
the level set; this cell structure is $G_L(S)$-equivariant and 
respects the decomposition into level sheets.  For $\sigma$ a
simplex of $L$, let $Y_\sigma$ be the subcomplex of $X^{(S)}_t$ 
comprising the union of all level sheets that correspond to $\sigma$.  
Each $Y_\sigma$ is a $G_L(S)$-subcomplex of $X^{(S)}_t$ and the 
stabilizer of each component of $Y_\sigma$ is a copy of $\zz^n$, 
where $n$ is the dimension of the simplex $\sigma$.  

Since $X_t=X^{(S)}_t$ is equal to the union of the subcomplexes
$Y_\sigma$, there is a surjective map of cellular chain complexes 
\[\bigoplus_{\sigma\in L} C_*(Y_\sigma)\rightarrow C_*(X^{(S)}_t).\]
Of course, some cells of $X_t$ arise in more than one $Y_\sigma$.  If
$c$ is an $n$-dimensional cell of $X_t$, then there is a unique
$n$-dimensional simplex $\sigma$ for which $c\in Y_\sigma$, and for
any $\tau\in L$, we have that $c\in Y_\tau$ if and only if $\sigma$ is
a face of $\tau$.  To compensate for the multiple counting of cells of
$X_t$, Davis and Okun introduce a double complex $C_{**}$, in which
$C_{i*}$ consists of the direct sum of one copy of $C_*(Y_{\sigma_0})$
for each chain $\sigma_0<\sigma_1< \cdots<\sigma_i$ of simplices of
$L$.  The indexing set for the summands of $C_{i*}$ is equivalently
the set of $i$-simplices of the barycentric subdivision $L'$ of $L$.
The bigraded homology of the double complex $C_{**}$, taken in the
horizontal direction, is zero for $i>0$ and is $C_j(X_t)$ for $i=0$;
essentially this follows from the fact that the Euler characteristic
of the star of any simplex of $L'$ is equal to~1.  

The case $S=\zz$ of the following theorem is~\cite[theorem~4.4]{do}.  

\begin{theorem}  
The $\ell^2$-Betti numbers of the space $X^{(S)}_t$ for $t\notin \zz$
are given in terms of the ordinary reduced Betti numbers of vertex
links in $L$ as follows:
\[
\beta_{(2)}^n(X^{(S)}_t;G_L(S))= \sum_{v\in L^0} 
\beta^n(\Star_L(v),\link_L(v)) = 
\sum_{v\in L^0} \overline{\beta}^{n-1}(\link_L(v)).\]
In the case when both $L$ and $\widetilde{L}$ are $\qq$-acyclic, 
these are the $\ell^2$-Betti numbers of the group $G_L(S)$, and 
so $\beta_{(2)}^n(G_L(S))$ does not depend on $S$.  
\end{theorem} 

\begin{proof}  
Consider the spectral sequences coming from the double cochain complex 
\[E_0^{ij} = \Hom_G(C_{ij},\cN(G)),\]
where $\cN(G)$ is the group von Neumann algebra of $G=G_L(S)$.  
The spectral sequence in which the horizontal differential is 
taken first collapses to give $H^*(\Hom_G(C_*(X_t),\cN(G)))$, 
which has the same von Neumann dimension (in L\"uck's extended
definition~\cite{lue}) as the $\ell^2$-cohomology of $X^{(S)}_t$.  
The spectral sequence in which the vertical differential is 
taken first is the one appearing in the statement
of~\cite[lemma~2.1]{do}.  
The $E_1$-page of this spectral 
sequence is identified with simplicial cochains on $L'$ 
with certain non-constant coefficients; the graded coefficient
module for the simplex $\sigma_0<\sigma_1<\cdots<\sigma_i$ in 
degree $j$ is $H^j(Y_{\sigma_0};\cN(G))$.  Since the $\ell^2$-cohomology
of $\zz^n$ vanishes for $n>0$ and is one copy of $\cc$ in 
dimension zero for $n=0$, it follows that $H^j(Y_\sigma;\cN(G))$ 
is equal to $\cN(G)$ for $j=\dim(\sigma)=0$ and otherwise 
has von Neumann dimension zero.  Up to modules of von Neumann 
dimension zero, the $E_1$-page is identified with the direct 
sum $\bigoplus_{v\in L^0}C^*(\Star_L(v),\link_L(v);\cN(G))$ 
concentrated in the line $j=0$.  It follows that $E_2^{i0}$ is 
isomorphic to $\bigoplus_{v\in L^0}H^i(\Star_L(v),\link_L(v);\cN(G))$, 
while other terms in the $E_2$-page have von Neumann dimension zero.   
Since modules of von Neumann dimension zero form a Serre class, no 
subsequent differential can affect the von Neumann dimensions, 
which gives the claimed result.  
\end{proof} 

Similarly, the case $S=\zz$ of the following theorem 
is~\cite[theorem~4.3]{do}.  In the case of 
cohomology with compact supports, the relevant spectral 
sequence still collapses at the $E_2$-page, but the $E_2$-page
contains many non-zero terms in the same total degree, which 
makes the statement more complicated.  In the statement, 
we write $G_\sigma$ for the stabilizer in $G_L(S)$ of a 
level sheet of type $\sigma$.  The group $G_\sigma$ is 
free abelian of rank equal to the dimension of $\sigma$ 
and it is well-defined up to conjugacy.  

\begin{theorem}\label{thmcompactco}  
For $t\notin \zz$, for any ring $R$ and for $n\geq 0$, there is a 
finite sequence of $RG_L(S)$-submodules of the compactly
supported cohomology $H_c^n(X^{(S)}_t;R)$ so that the associated 
graded module is
\[{\mathrm{Gr}}H_c^n(X^{(S)}_t;R)\cong \bigoplus_{\sigma\in L}
H^{n-\dim{\sigma}}(\Star_L(\sigma),\link_L(\sigma);R)\otimes_R 
RG_L(S)/G_\sigma.\] 
In the case when $L$ and $\widetilde{L}$ are 
both $R$-acyclic, note that $H_c^n(X^{(S)}_t;R)\cong H^n(G;RG)$.  
\end{theorem} 

In the statement the filtration grading is by $n-\dim(\sigma)$, so 
that the contribution coming from vertices of $L$ is a quotient 
of the compactly supported cohomology $H^n_c$ and simplices of 
largest possible dimension contribute a submodule of $H^n_c$.  

\begin{proof} 
In this case, one considers the spectral sequences arising from 
the double cochain complex 
\[E_0^{ij} = \Hom_G(C_{ij},RG).\]
As before, the spectral sequence in which the horizontal differential
is taken first has $E_1^{ij}=0$ for $i>0$, and $E_2^{0j}$ is isomorphic
to $H^j_c(X^{(S)};R)$.  In the spectral sequence in which the vertical 
differential is taken first $E_1^{ij}$ is the $i$th simplicial cochains 
on $L'$ with graded non-constant coefficients in which the simplex 
$\sigma_0<\cdots<\sigma_i$ is assigned $H^j_c(Y_{\sigma_0};R)$ in 
degree $j$.  As in~\cite[theorem~4.3]{do}, since $Y_{\sigma_0}$ is a 
disjoint union of level sheets, and $G_{\sigma_0}$ is free abelian of 
rank $\dim(\sigma_0)$, it follows that $H^j_c(Y_{\sigma_0};R)$ is 
equal to $RG/G_{\sigma_0}$ for $j= \dim(\sigma)$ and is zero 
otherwise.  The computation of the $E_2$-page and the fact that 
the higher differentials are trivial is more delicate than in the 
previous theorem, but it works in exactly the same way for $G_L(S)$ 
as for $BB_L$; essentially it follows from the fact that the 
map $H^j_c(\rr^n;R)\rightarrow H^j_c(\rr^m;R)$ is the zero map 
for each $j$ whenever $\rr^m$ is a proper subspace of $\rr^n$.  
Given this, the $E_1$-page splits as a direct sum of cochain 
complexes and 
\[E_2^{ij}\cong \bigoplus_{\dim(\sigma)=j} 
H^i(\Star_L(\sigma),\link_L(\sigma);R)\otimes_R 
RG_L(S)/G_\sigma.\] 
The claim follows.  
\end{proof} 

\begin{corollary} \label{corduality} 
The following are equivalent.   
\begin{itemize} 
\item 
For each $S\subseteq\zz$, $G_L(S)$ is an $n$-dimensional 
duality group.  
\item 
$L$ and $\widetilde{L}$ are acyclic and for every simplex 
$\sigma$ of $L$, the reduced homology of $\Lk_L(\sigma)$ 
is free and is non-zero only in degree $n-1-\dim(\sigma)$.  
\end{itemize} 
\end{corollary} 

\begin{proof} 
A duality group is a group $\Gamma$ of type $FP$ such that 
$H^*(\Gamma;\zz\Gamma)$ is torsion-free and concentrated 
in a single degree~\cite[theorem~10.1]{brown}.  The claim 
follows from Theorem~\ref{thmb} and Theorem~\ref{thmcompactco}.  
\end{proof}

\section{Poincar\'e duality groups} 

As in~\cite{bieri,brown}, define a Poincar\'e duality group to be a 
group $\Gamma$ of type $FP$ such that $H^*(\Gamma;\zz\Gamma)$ 
is isomorphic (as abelian group) to $\zz$.  If 
$H^n(\Gamma;\zz\Gamma)\cong \zz$, then $\Gamma$ is an 
$n$-dimensional Poincar\'e duality group.  The fundamental 
group of a closed aspherical $n$-manifold is a finitely 
presented Poincar\'e duality group.  

In~\cite{davispd}, Davis constructed the first non-finitely 
presentable Poincar\'e duality groups.  More precisely, he 
applied his `reflection group trick' to 
show that every group of type $FH$ is a retract of a Poincar\'e 
duality group; see also~\cite{davisone,davisbook,mess}.  This theorem 
combined with the Bestvina-Brady 
examples gave the new Poincar\'e duality groups.  Similarly
Davis's theorem applied to the groups $G_L(S)$ gives rise
to an uncountable family of Poincar\'e duality groups. 
The following stronger result can be proved in a similar way.  

\begin{theorem}\label{thmuncpd}  
For each $n\geq 4$ there is a closed aspherical $n$-manifold $M$ 
that admits a family of acyclic regular coverings $M(S)$ indexed 
by the subsets $S$ of $\zz$, in such a way that for $S\subseteq T
\subseteq \zz$, $M(S)$ is a covering of $M(T)$.  Members of this 
family give rise to uncountably many non-isomorphic groups of deck 
transformations.  Each such group of deck transformations is an 
$n$-dimensional Poincar\'e duality group.  
\end{theorem} 

Before beginning the proof, we recall some of the material concerning
Coxeter groups that we will use; for more details
see~\cite{davisbook}.  Given a flag complex $K$, the right-angled
Coxeter group $W_K$ is the group generated by elements in bijective
correspondence with the vertices of $K$, subject only to the relations
that each generator has order two and that generators corresponding to
vertices that span an edge commute with each other.  If $L$ is a full
subcomplex of $K$, the natural map on generating sets embeds $W_{L}$
as a subgroup of $W_K$; in particular each $(i-1)$-simplex $\sigma$ of
$K$ gives rise to a subgroup $W_\sigma\leq W_K$ which is isomorphic to
the direct product $(C_2)^i$.  A \emph{special subgroup} of $W_K$ is either
the trivial subgroup or a subgroup $W_\sigma$ for $\sigma$ a simplex
of $K$; it is convenient to view the empty set as the unique
$-1$-simplex in $K$ and define $W_\emptyset$ to be the trivial
subgroup of $W_K$.  

The Davis complex $\Sigma_K$ is a CAT(0) cubical complex 
with a properly discontinuous cellular action of $W_K$.  
It is slightly easier to define first the barycentric 
subdivision of $\Sigma_K$.  This is the simplicial complex obtained
as the realization of the poset whose elements are the cosets of 
the special subgroups, ordered by inclusion.  For any 
$(i-1)$-simplex $\sigma$ of $K$ and any $w\in W$, the subposet 
of cosets contained in $wW_\sigma$ is isomorphic to the poset 
of faces of an $i$-cube, and these are the $i$-cubes of $\Sigma_K$.  
The complex $\Sigma_K$ is simply-connected, $W_K$ acts freely 
transitively on its vertex set, and the link of each vertex is $K$.  
It follows that $\Sigma_K$ is CAT(0).  All cube stabilizers are 
finite, and so any torsion-free subgroup of $W_K$ acts freely on 
$\Sigma_K$.  

\begin{lemma} \label{lemmfld} 
Let $X$ be a finite 2-complex.  For any $n\geq 4$ there is a 
compact triangulable manifold $V$ homotopy equivalent to $X$.  
\end{lemma} 

\begin{proof} 

See \cite[theorem~11.6.4]{davisbook}. 
\end{proof} 

\begin{lemma}\label{lemcx}
There is a finite aspherical 2-complex $X$ that admits a family of
acyclic regular covers $X(S)$ indexed by the subsets $S$ of $\zz$, 
with functoriality as in the statement of Theorem~\ref{thmuncpd}.  
\end{lemma}

\begin{proof} 
Let $L$ be a flag 2-complex that is acyclic and aspherical but 
not simply-connected, and satisfies nlcp; for example the 2-complex
described in Section~2.  
Let $X_t$ be a non-integer level set in the cubical complex $X_L$, and
define $X$ to be the finite 2-complex $X_t/BB_L$.  By
Corollary~\ref{corgzero}, the fundamental group of $X$ is
$G_L(\emptyset)$.  By Corollary~\ref{corregcov}, for each $S\subseteq
\zz$, $X$ has a regular covering space $X^{(S)}_t$, with $G_L(S)$ as
its group of deck transformations.  Define $X(S)$ to be this covering 
space of $X$.  By Corollary~\ref{brancov} this construction is 
functorial in $S$ in the following sense.  
For $S\subseteq T\subseteq
\zz$, the natural surjective homomorphism $G_L(S)\rightarrow G_L(T)$
can be used to define an action of $G_L(S)$ on $X(T)$, and there is a
regular covering map $X(S)\rightarrow X(T)$ which is
$G_L(S)$-equivariant.  Moreover, for each $S\subseteq \zz$,
$X(S)/G_L(S)=X$.
\end{proof}

\begin{proof} 
(of Theorem~\ref{thmuncpd}) 
Let $L$ be a flag complex as used in the proof of Lemma~\ref{lemcx},
and let $X$ be the aspherical 2-complex constructed in Lemma~\ref{lemcx}
with fundamental group $G_L(\emptyset)$.  Now fix $n\geq 4$, and 
let $V$ be a compact $n$-manifold homotopy 
equivalent to $X$ as in Lemma~\ref{lemmfld}.  Let $K$ be the 
barycentric subdivision of a triangulation of the boundary 
$\partial V$ of $V$.  
The fundamental group of $V$ is isomorphic to $G_L(\emptyset)$.  
For $S\subseteq \zz$, let $V(S)$ denote the cover of $V$ that 
is homotopy equivalent to $X(S)$, and similarly let $K(S)$ be 
the induced triangulation of $\partial V(S)$.  Thus $G_L(S)$ 
acts freely on $V(S)$ and on $K(S)$, with quotient spaces 
$V$ and $K$ respectively, and these spaces enjoy the same 
functoriality as the spaces $X(S)$.  Since $K$ is a barycentric
subdivision, it admits a simplicial map to an $(n-1)$-simplex 
$\Delta$ that is injective on each simplex; the vertices of 
$\Delta$ are labelled $0,1,\ldots,n-1$, and the map sends 
the barycentre of a simplex $\tau$ to $\dim(\tau)$.  There 
is also an induced map $K(S)\rightarrow \Delta$ for each
$S\subseteq\zz$.  

Now let $W$ denote the finitely-generated Coxeter group $W_K$, 
and for $S\subseteq \zz$, let $W(S)$ be the Coxeter group 
$W_{K(S)}$.  For $S\subseteq T$, the maps $K(S)\rightarrow K(T)$
and $K(S)\rightarrow K$ induce surjective group homomorphisms 
$W(S)\rightarrow W(T)$ and $W(S)\rightarrow W$, defining a 
functor.  The maps to $\Delta$ define surjective homomorphisms 
$W(S)\rightarrow W_\Delta$ and $W\rightarrow W_\Delta$, and 
we define $W'(S)$ and $W'$ to be the kernels of these maps.  
Thus $W'(S)$ and $W'$ are normal subgroups of index $2^n$ in 
$W(S)$ and $W$ respectively.  

Let $\Sigma$ denote the Davis complex $\Sigma_K$, and for 
$S\subseteq \zz$, let $\Sigma(S)$ denote the Davis complex 
$\Sigma_{K(S)}$ associated to the Coxeter group $W(S)$.  
For $S\subseteq T$ the surjective homomorphisms $W(S)\rightarrow W(T)$ 
and $W(S)\rightarrow W$ induce equivariant maps $\Sigma(S)\rightarrow 
\Sigma(T)$ and $\Sigma(S)\rightarrow \Sigma$.  Since each 
special subgroup of $W$ or $W(S)$ maps isomorphically to a 
subgroup of $W_\Delta\cong C_2^n$, the group $W'(S)$ acts 
freely on $\Sigma(S)$ and $W'$ acts freely on $\Sigma$.  

Now $G_L(S)$ acts freely on $K(S)$, which induces an action 
by automorphisms of $G_L(S)$ on $W(S)$ (one might describe 
this as an action by `diagram automorphisms').  Since the 
$G_L(S)$-action preserves the set of special subgroups 
there is an induced action of the semi-direct product 
$W(S)\semi G_L(S)$ on $\Sigma(S)$.  To simplify notation, 
let $J(S)=W(S)\semi G_L(S)$.  For $S\subseteq T$ 
there is a surjective group homomorphism $J(S)\rightarrow J(T)$, 
and the natural map 
$\Sigma(S)\rightarrow \Sigma(T)$ is equivariant for 
this map.  Since $K=K(S)/G_L(S)$, there is also a surjective 
group homomorphism $J(S)\rightarrow W$, and 
the natural map $\Sigma(S)\rightarrow \Sigma$ is equivariant
for this homomorphism.  

Now let $H(S)= W'(S)\semi G_L(S)$, so that $H(S)$ is an 
index $2^n$ normal subgroup of $J(S)$, and let $H= W'$.  This  
$H$ acts freely cocompactly on~$\Sigma$.  The action of 
$H(S)\leq J(S)$ on $\Sigma(S)$ is free except that each vertex 
of $\Sigma(S)$ has stabilizer a $J(S)$-conjugate of the standard
copy of $G_L(S)\leq H(S)$.  Nevertheless, the action is 
cocompact, and the natural maps induce isomorphisms of 
finite cube complexes $\Sigma(S)/H(S)\cong \Sigma/H$.  
Since the actions of $H(S)$ and $H(T)$ are free except 
at the vertices, 
the equivariant maps $\Sigma(S)\rightarrow \Sigma(T)$ 
and $\Sigma(S)\rightarrow \Sigma$ are branched coverings, 
with branching only at the vertices.  

The link of each vertex of $\Sigma(S)$ is a copy of the
$(n-1)$-manifold $K(S)$ and the link of each vertex of $\Sigma$ is a
copy of $K$.  From this it follows that $\Sigma(S)$ and $\Sigma$ are
locally homeomorphic to $\rr^n$, except at their vertices.  Remove
from $\Sigma(S)$ the open balls of radius $1/4$ with centres the
vertices of $\Sigma(S)$.  These balls have disjoint closures, and the
resulting space is a $J(S)$-manifold whose boundary is equivariantly
homeomorphic to $J(S)\times_{G_L(S)}K(S)$.  Define $M(S)$ by attaching
$J(S)\times_{G_L(S)}V$, which has the same boundary.  Let
$\overline{M}$ be defined similarly by removing open balls around the
vertices from $\Sigma$ and identifying the boundary of the resulting
manifold with the boundary of $W\times V$.  Now define $M$ to be
$\overline{M}/H$.  Since $H$ acts freely cocompactly on
$\overline{M}$, $M$ is a closed manifold, built by identifying the
boundaries of $2^n$ copies of $V$.  Since each $V(S)$ is acyclic, the
homology of $M(S)$ is isomorphic to the homology of $\Sigma(S)$, and
so each $M(S)$ is an acyclic manifold.  Since $V(\emptyset)$ is simply
connected, it follows that $M(\emptyset)$ is simply connected too, and
hence $M(\emptyset)$ is contractible.  For each $S$, $H(S)$ acts
freely cocompactly on the acyclic manifold $M(S)$ with $M(S)/H(S)\cong
M$, and for $S\subseteq T$ the map $M(S)\rightarrow M(T)$ is an
equivariant covering map.  The universal covering space of $M$ is the
contractible manifold $M(\emptyset)$, which implies that $M$ is
aspherical as claimed.

It remains to show that there are uncountably many isomorphism types
of group $H(S)$.  This follows from the fact that there
are uncountably many isomorphism types of group $G_L(S)$ since 
$G_L(S)$ is a subgroup of $H(S)$; alternatively one can apply 
the invariant $\cR$ directly to a suitable sequence of elements
of $H(S)$.  
\end{proof} 

\section{$FP_2$ via presentations} 

In this section we briefly outline an alternative approach to 
the groups $G_L(S)$, along the lines of the treatment of 
$BB_L$ given in~\cite{DL}.  This uses the presentation 
$P_L(S,\Gamma)$ and the geometry of $L$, but not $X_L^{(S)}$. 
These results represent our first approach to the groups 
$G_L(S)$ and they motivated our discovery of $X_L^{(S)}$.  

As usual, $L$ is a finite connected flag complex, and $\Gamma$ a finite 
collection of directed edge loops in $L$ such that $\pi_1(L)$ 
can be killed by attaching discs to these loops.  It will be 
useful to reduce the size of the sets of generators and relations, 
so we fix a choice of orientation for each edge $a$ of $L$, and 
take the presentation $P'_L(S,\Gamma)$ in which the generator 
$\overline{a}$ corresponding to the opposite orientation is 
eliminated, together with the length two relator $a\overline{a}$.  
To avoid having to write $(a_1^{\epsilon_1},\ldots,a_l^{\epsilon_l})$ 
for a directed edge loop, we make the convention that a single letter
such as $a_i$ may stand for either one of our preferred generators or 
its inverse.  

\begin{lemma} \label{rellem} 
For any directed edge loop $\beta=(b_1,\ldots,b_m)$ in $L$, and any $n\in S$, 
the relation $b_1^nb_2^n\cdots b_m^n=1$ is a consequence of the 
relations in $P'_L(S,\Gamma)$.  
\end{lemma} 

\begin{proof} 
Identical to the proof in~\cite{DL}, which considered only the case 
$S=\zz$; we give a brief account mainly to establish notation.  
If $\beta$ is null-homotopic then there is a simplicial 
map from a disc $D^2$ to $L$ so that the boundary of the disc is 
the loop $\beta$.   This gives rise to a van Kampen diagram $D$ in 
the group presented by $P'_L(S,\Gamma)$, with $b_1^n\cdots b_m^n$ 
as its boundary word.  The interior regions of $D$ are triangles, 
each edge is labelled $a^n$ for some directed edge $a$ of $L$, or 
possibly $1=1^n$, and the labels around each triangle are either 
$a^nb^nc^n$ for a directed triangle $(a,b,c)$ of $L$, or $1^n1^n1^n$ 
or $1^na^na^{-n}$.  These relations are all consequences of the 
short relations in $P'$.  Similarly, if $\beta$ is homotopic to 
some $\gamma=(a_1,\ldots,a_l)\in \Gamma$, a simplicial homotopy 
can be used to build
a van Kampen diagram with $b_1^n\cdots b_m^n$ as its outer boundary, 
in which all regions are triangles with boundaries as above 
except for one larger disc with boundary word $a_1^na_2^n\cdots a_l^n$.  
An arbitrary edge loop $\beta$ can be expressed as a concatenation 
of loops of the above types.  
\end{proof} 

By the above lemma, the group $G_L(S)$ defined by the presentation
$P'_L(S,\Gamma)$ does not depend on choice of $\Gamma$.  Now let $\cC$
be the Cayley 2-complex of the presentation $P'$, and let $\cC_t$ 
be the subcomplex consisting of the 1-skeleton and the 2-cells coming 
from triangle relators.  Let $C_2\rightarrow
C_1\rightarrow C_0$ be the cellular chain complex for $\cC$, with 
$d:C_2\rightarrow C_1$ being the boundary map.  Let $C_2^t$ denote the
$G_L(S)$-submodule of $C_2$ coming from the triangle relators and let
$C_2^l$ denote the $G_L(S)$-submodule coming from the `long relators'
$a_1^n\cdots a_l^n$ for $n\in S$ and $\gamma\in \Gamma$.  Note that
$C_2^t$ is finitely generated, and that $C_2=C_2^l\oplus C_2^t$.

\begin{theorem} 
In the case when $\pi_1(L)$ is perfect, $d(C_2^l)\leq d(C_2^t)$.  
Hence in this case $G_L(S)$ is type $FP_2(\zz)$.  
\end{theorem} 

\begin{proof} 
Fix $n\in S$ and let $\gamma=(a_1,a_2,\ldots,a_l)$ be a directed edge
loop in $\Gamma$.  By hypothesis, $\gamma$ represents the
trivial element of $H_1(L)$, and so we can find a compact orientable
surface $F$, with boundary $\partial F$ a single circle, and a
simplicial map from $F$ to $L$ so that $\gamma$ is the image of
$\partial F$.  This gives rise to a `generalized van Kampen diagram'
in the presentation $P'$, which we shall also call $F$, by a slight
abuse of notation.  Thus we have a labelling of the edges of $F$ by
elements $a^n$ or $1=1^n$, so that the word around $\partial F$ spells
out $a_1^n\cdots a_l^n$ and so that each triangle is labelled by a
relator that is a consequence of the triangle relators.  Fix as
basepoint $p\in F$ the vertex of $\partial F$ that corresponds to the
start of the word $a_1^n\cdots a_l^n$.  Note also that any edge path
in $F$ is mapped to an edge path in $L$.  For each vertex $v$ of $F$
and each $n\in S$, define $g(v,n)\in G_L(S)$ by 
$g(v,n)=b_1^n\cdots b_m^n$, where
$(b_1,\ldots,b_m)$ is any directed edge path in $L$ that is the image
of a directed edge path in $F$ from $p$ to $v$.  By
Lemma~\ref{rellem}, $g(v,n)$ depends only on $v$~and~$n\in S$ but
not on the choice of edge path.  Orient the triangles of $F$
consistently and pick a base point $v(\sigma)$ for each triangle
$\sigma$ of $F$.

Now fix a base vertex $x$ in the Cayley 2-complex $\cC$.  
For each triangle $\sigma$ of $F$, the word $a^nb^nc^n$ 
around the boundary of $\sigma$ is a consequence of the 
triangle relators.  Hence there is a disc $\phi(\sigma)$
in $\cC^t$ that bounds the path that starts at the vertex 
$g(v(\sigma),n).x$ and follows the edge path $a^nb^nc^n$. 
The union of all of these discs can be used to define a 
map from $F$ to $\cC^t$ such that $\partial F$ maps to 
edge path starting at $x$ and given by the word 
$a_1^n\cdots a_l^n$.  It follows that the attaching map for 
each 2-cell is contained in $d(C_2^t)$ and so $d(C_2^l)\subseteq 
d(C_2^t)$ as claimed.  
\end{proof} 

For any group presentation, a van Kampen diagram can be viewed 
as describing a map from a disc into the Cayley 2-complex.  However, 
a `generalized van Kampen diagram' on a surface $F$ does not 
in general define a map from $F$ to the Cayley 2-complex; the 
fact that this works for the presentation $P'_L(S,\Gamma)$ 
relies crucially on Lemma~\ref{rellem}.

\section{The homotopy type of level sets} 

In the original Bestvina-Brady paper~\cite{BB}, some parts of the 
main theorem depend upon a study of the homotopy types of level 
sets $X_{[a,b]}=f^{-1}([a,b])\subseteq X_L$, which comprises 
section~8 of~\cite{BB}.  This is not essential to our argument, 
which is closer to that of~\cite{buxgon}.  Nevertheless, we 
state the analogue of theorem~8.6 of~\cite{BB}.  

\begin{theorem} 
Suppose that $\pi_1(L)$ is finite.  For any $a\leq b\in \rr$, 
the level set $X^{(S)}_{[a,b]}$ is homotopy equivalent to a 
wedge of copies of $L$, indexed by the vertices of $X^{(S)}$ 
of height in $S-[a,b]$ and copies of $\widetilde{L}$, indexed
by the vertices of height in $(\zz-S)-[a,b]$.  
\end{theorem} 

\begin{proof} 
Identical to the proof of~\cite[thm~8.6]{BB}.  For each vertex $v$ of
$X^{(S)}$, let $K_v$ be the union of all sheets containing~$v$.
Analogues of lemmas 8.1--8.5 of~\cite{BB} hold; in particular each
$K_v$ is a convex subset of $X^{(S)}$.  The hypothesis that $\pi_1(L)$
is finite implies that $X^{(S)}$ is locally finite, which allows us to
enumerate its vertex set as $v=v_1,v_2,v_3,\ldots$ in such a way that
$d(v,v_i)\leq d(v,v_j)$ whenever $i<j$.  Such an enumeration plays an
important role in the proof of~\cite[lemma~8.3]{BB} and its analogue
for $X^{(S)}$.
\end{proof} 

We believe that the above theorem should hold without the hypothesis  
that $\pi_1(L)$ be finite.  This would give an alternative 
treatment of some of our results.

\section{Closing remarks} 

It would be easy to generalize some of our results to the case in 
which $\widetilde{L}$ is replaced by some other regular covering 
$\widehat{L}$ of $L$.  To do this, embed $L$ as a full subcomplex 
of some $M$ with nlcp in such a way that the 
induced map $\pi_1(L)\rightarrow \pi_1(M)$ is a surjection 
with kernel $\pi_1(\widehat{L})$.  Given such an $M$, 
the statement and (`general case') proof of Theorem~\ref{thmexess} 
go through unchanged.  

Ignoring our results concerning higher finiteness properties, our
main theorem gives a new proof that there are uncountably many 
finitely generated groups.  The simplest family of such groups 
arising from our theorem is given by the presentations 
\[\langle a,b,c,d \,\,\colon \,\, a^nb^nc^nd^n = 1 \,\,\forall 
n\in S\rangle.\] 
These are the presentations $P_L(\Gamma,S)$ for $0\in S\subseteq \zz$
in the case when $L$ is a circle, triangulated as the boundary of a 
square.  

We reiterate the remark made after Proposition~\ref{propsfinfhr}: 
In the case when $\zz-S$ is finite and $L$ is $R$-acyclic, we know  
that $G_L(S)$ is $FP(R)$, but we do not know whether $G_L(S)$ is 
necessarily $FH(R)$.  

In an earlier version of this article, we drew attention to two
questions related to this work.  Does every group of type $F_n$ embed
in a group of type $F_{n+1}$?  Does every group of type $FP_n(R)$
embed in a group of type $FP_{n+1}(R)$?  These questions are open for
each $n\geq 2$.  We have recently shown that every countable group
embeds in a group of type $FP_2(\zz)$~\cite{sgfp2}, resolving the case
$n=1$ of the second question.

\leftline{\bf Author's address:}

\obeylines

\smallskip
{\tt i.j.leary@soton.ac.uk}

\smallskip
School of Mathematical Sciences, 
University of Southampton, 
Southampton,
SO17 1BJ

\end{document}